\documentclass[11pt,reqno]{amsart}

\usepackage{amsmath, amsfonts, amssymb, amsthm}
\usepackage[top=1in, left=1in, right=1in, bottom=1in]{geometry}
\usepackage{microtype}
\usepackage{mathtools, mathrsfs}
\usepackage{etoolbox}
\usepackage{enumitem}
\usepackage[usenames,dvipsnames]{xcolor}
\usepackage{comment}
\usepackage{hyperref}

\let\originalleft\left
\let\originalright\right
\renewcommand{\left}{\mathopen{}\mathclose\bgroup\originalleft}
\renewcommand{\right}{\aftergroup\egroup\originalright}
\everymath{\displaystyle}
\allowdisplaybreaks

\newtoggle{draftmode}
\togglefalse{draftmode}		

\newcommand{\comD}[1]{\iftoggle{draftmode}{{\textcolor{red}{\textbf{\textit{#1}}}}}{\ignorespaces}}		
\newcommand{\comT}[1]{\iftoggle{draftmode}{{\textcolor{ForestGreen}{\textbf{\textit{#1}}}}}{\ignorespaces}}	

\newcommand{\R}{\mathbb{R}}							
\newcommand{\BA}{\operatorname{\mathbf{BA}}}		
\newcommand{\Z}{{\mathbb{Z}}}						
\newcommand{\Q}{{\mathbb{Q}}}						
\newcommand{\GL}{\operatorname{GL}}					
\newcommand{\SL}{\operatorname{SL}}					
\newcommand{\M}{{\mathcal{M}}}						
\newcommand{\MK}{{\M_K}}							
\newcommand{\N}{{\mathbb{N}}}						
\newcommand{\BO}{{\mathcal{O}}}						
\newcommand{\CO}{{\mathbb{C}}}						
\newcommand{\Hc}{{\mathcal{H}}}						%
\newcommand{\Lc}{{\mathcal{L}}}						%
\newcommand{\abwin}{$(\alpha, \beta)$-winning}		
\newcommand{\awin}{$\alpha$-winning}				

\newcommand{\pbar}{\overline{p}}
\newcommand{\qbar}{\overline{q}}

\newcommand{\iv}[2][]{\ifstrempty{#2}{\iota_{#1}}{\iota_{#1} \left( {#2} \right)}}	
\newcommand{\vect}[1]{\vec{\mathbf{#1}}}									
\newcommand{\vvect}[2]{\begin{matrix} {#1} \\ {#2} \end{matrix}}	
\newcommand{\pvect}[2]{\begin{pmatrix} {#1} \\ {#2} \end{pmatrix}}	
\newcommand{\abs}[1]{\left| {#1} \right|}							
\newcommand{\norm}[1]{\left\| {#1} \right\|}						
\newcommand{\prt}[1]{\left( {#1} \right)}							
\newcommand{\cbrk}[1]{\left\{ {#1} \right\}}						
\newcommand{\height}[1]{\operatorname{H}\left( {#1} \right)}		
\newcommand{\hplane}[2]{\mathcal{L} \prt{ {#1}, {#2} }}				
\newcommand{\dist}[1]{\ifstrempty{#1}{\operatorname{dist}}{{\operatorname{dist}} \left( {#1} \right)}}	

\numberwithin{equation}{section}
\newtheorem{Thm}{Theorem}[section]
\newtheorem{Prop}[Thm]{Proposition}
\newtheorem{Cor}[Thm]{Corollary}
\newtheorem{Lem}[Thm]{Lemma}

\theoremstyle{definition}
\newtheorem{Rem}[Thm]{Remark}

\newtheorem{Eg}[Thm]{Example}

\newcommand{\eq}[2]{\ifstrempty{#1}{\begin{equation*} {#2} \end{equation*}}{\iftoggle{draftmode}{ {\tt ~[#1]~} }{\ignorespaces} \begin{equation} \label{eq:#1} {#2} \end{equation}}}
\newcommand{\eqr}[1]{\eqref{eq:#1}}
\newcommand{\lab}[1]{\label{#1} \iftoggle{draftmode}{{\tt [#1]}}{\ignorespaces}}

\newcommand {\new}[1]   {\iftoggle{draftmode}{{\leavevmode\color{blue}{#1}}}{{#1}}}

\newcommand{\x}{{\mathbf{x}}}	
\newcommand{\y}{{\mathbf{y}}}
\newcommand{\z}{{\mathbf{z}}}
\newcommand{\da}{Diophantine approximation}
\newcommand{\hd}{Hausdorff dimension}
\newcommand{\B}{{\mathbf{B}}}
\newcommand{\supp}{\operatorname{supp}}		

\begin{document}

\title[Badly approximable $S$-numbers and absolute Schmidt games]{Badly approximable $S$-numbers \\
and 
absolute Schmidt games}

\author{Dmitry Kleinbock}
\address{Brandeis University, Waltham MA
02454-9110 {\tt kleinboc@brandeis.edu}}

\author{Tue Ly}
\address{Brandeis University, Waltham MA
02454-9110 {\tt lntue@brandeis.edu}}

\date{August 2014}

\thanks{The first-named author was supported by NSF grant DMS-1101320. The authors' stay at MSRI was supported in part by NSF grant no. 0932078 000.}

\begin{abstract} Let $K$ be a number field, let $S$ be  
the set of all normalized, non-conjugate Archimedean valuations of $K$, and let $K_{S} = \prod_{v \in S} K_v$ be the Minkowski space associated with $K$. 
	We strengthen 
	 recent results of~\cite{EsdahlKristensen10} and \cite{EinsiedlerGhoshLytle13} 
	by showing that the set of 
	badly approximable elements of $K_S$ 	
	is $\Hc$-absolute winning for a certain family of subspaces of $K_{S}$. \end{abstract}

\maketitle

\tableofcontents

\section{Introduction}
\lab{Sec:Intro}
In the classical theory of Diophantine approximation, a number $x \in \R$ is called \emph{badly approximable} if there exists $c = c_x > 0$ such that for every $p \in \Z$, $q \in \N$:
\eq{BA in R}{\abs{x - \frac{p}{q}} > \frac{c}{q^2}, \text{ or equivalently, } \abs{qx - p} > \frac{c}{q}.}
The set of 
badly approximable numbers is small in that its Lebesgue measure is 0.  Nevertheless, this set is \emph{thick}, which means that its intersection with any open set in $\R$ has full Hausdorff dimension.  In fact, the intersection of any countable collection of translations of this set is still thick.  This remarkable result was first proved by W. Schmidt~\cite{Schmidt66} by showing that the set of badly approximable numbers is a winning set of an infinite topological game, which will be called \emph{Schmidt game}.  In Schmidt game, two players (Alice and Bob) alternatively choosing nested balls and Alice's goal is to steer the intersection point to the target set, which is called a \emph{winning set} of the game if Alice has a winning strategy (see~\S\ref{Subsec:Schmidt game}).  Inspired by Schmidt, C. McMullen~\cite{McMullen10} introduced a variation of Schmidt game, in which Alice can only delete neighborhoods of points, and proved that the set of badly approximable numbers is winning on this game.  McMullen's version of the game is called \emph{absolute game}, and its winning sets are called \emph{absolute winning}.  Absolute winning is a stronger property than winning and is preserved 
by quasi-symmetric maps (see~\S\ref{Subsec:H-absolute game}).

Our goal is to generalize the above results into the setting of algebraic number fields.  To be more precise, let $K$ be a number field of degree $d$ and 
let $S$ be 
the set of all normalized, non-conjugate Archimedean valuations of $K$.  
If $v$ is an element of $S$, the corresponding embedding of $K$ into its completion $K_v$ is denoted by $\iv[v]{}$.
We will replace $\R$ by the \emph{Minkowski space} associated with $K$, defined by:
\eq{Kinf}{K_{S} := \prod_{v \in S} K_v \cong K \otimes_\Q \R \cong \R^d;}
elements $\x = (x_v)_{v \in S} \in K_{S}$ will be referred to as \emph{$S$-numbers}.  The diagonal embedding of $K$ into $K_S$ is denoted by $\iv[S]{}$:
\eq{}{\iv[S]{} : K \to K_S, r \mapsto \prt{\iv[v]{r}}_{v \in S}.}

We also equip
$K_S$ with the $K_S$-valued inner product:
\eq{}{ \x \cdot \y = \prt{x_v y_v}_{v \in S} }
 and the sup norm
\eq{}{ \norm{\x} = \max_{v \in S} \abs{x_v} }
for every $\x = \prt{x_v}_{v \in S}$ and $\y = \prt{y_v}_{v \in S}$ in $K_S$.

We let $\BO$ be the ring of integers of $K$; note that
$\iv[S]{\BO}$ is a lattice in $K_S$.

Following~\cite{EinsiedlerGhoshLytle13}, we 
will say that $\x \in K_{S}$ is \emph{badly approximable} if there exists $c > 0$ 
such that for every $p \in \BO$, $q \in \BO \smallsetminus \cbrk{0}$:
\eq{BA in K}{\norm{\iv[S]{q}\cdot \x +\iv[S]{p}} = \max_{v \in S} \abs{\iv[v]{q} x_v + \iv[v]{p}} > c \prt{\max_{v \in S} \abs{\iv[v]{q}}}^{-1} = c \|\iv[S]{q}\|^{-1}.}
or equivalently,
\eq{BA in K-inf}{\inf \cbrk{ \norm{\iv[S]{q}} \cdot \norm{\iv[S]{q} \cdot \x + \iv[S]{p}} : p, q \in \BO, (p, q) \neq 0 } > 0.}

Let us denote the set of badly approximable 
${S}$-numbers by $\BA_{K}$.  In particular, when $K = \Q$,~\eqr{BA in K} is the same as \eqr{BA in R}, and we come back to the classical case.  Measure-wise, $\BA_K$ is small, that is, it can be shown to have Lebesgue measure zero,  see~\cite{EinsiedlerGhoshLytle13} or Remark \ref{erg} below.  
Explicit examples of badly approximable ${S}$-numbers for arbitrary number field $K$ were constructed by Burger~\cite{Burger92}, and when $K$ is real quadratic or totally complex quartic, Hattori~\cite{Hattori07} showed that $\BA_K$ is uncountable.  Note that those authors used different definitions of badly approximable numbers, but their examples in fact belong to $\BA_K$.  The first result regarding the winning property of $\BA_K$ beyond the classical case was proved by Esdahl-Schou and Kristensen in~\cite{EsdahlKristensen10} as follows:

\begin{Thm}[{\cite[Lemma 4]{EsdahlKristensen10}}] \lab{Thm:ESK}
	Let $K$ be an imaginary quadratic number field with class number 1, i.e., $K = \Q(\sqrt{-D})$, where $D \in \cbrk{1, 2, 3, 7, 11, 19, 43, 67, 163}$, and $Y \subseteq \CO \cong K_S$ be a compact set supporting an Ahlfors regular measure $\mu$.	Then $\BA_K \cap~Y$ is winning for the Schmidt game playing on $Y$.
\end{Thm}

We recall that $\mu$ is said to be  \emph{Ahlfors regular} if  there exist constants $a, b, \delta, r_0 > 0$ such that for any $z \in Y$ and any $0 < r \leq r_0$,
	\eq{Ahlfors}{ar^{\delta} \leq \mu \big(\B(z, r)\big) \leq br^{\delta},}
	where $\B(z, r)$ is the closed ball centered at $z$ of radius $r$.  It follows that $Y = \supp\,\mu$ has \hd\ $\delta$, and one knows, see e.g.\ \cite[Lemma 3]{EsdahlKristensen10}, that any subset of $Y$ which  winning on  $Y$ is thick.

A few years later, Einsiedler et al.~\cite{EinsiedlerGhoshLytle13} considered  an arbitrary number field $K$ and showed that the intersections of $\BA_K$ with certain smooth curves in $K_{S}$ are 
winning. The following notation is needed to state their result: for a   subset $T \subseteq S$,  
denote by $T_\R$ and $T_\CO$ 
the sets of 
valuations 
in $T$ whose corresponding $K$-completions are isomorphic to $\R$ 
or $\CO$ respectively:
\eq{}{T_\R := \cbrk{v \in T: K_v \cong \R} \text{~\quad and \quad} T_\CO := \cbrk{v \in T: K_v \cong \CO}.}

\begin{Thm}[{\cite[Theorem 1.1]{EinsiedlerGhoshLytle13}}] \lab{Thm:EGL}
	Let $\phi : [0, 1] \to K_{S}$ be a continuously differentiable map.  For any $x\in[0, 1]$, define
	$$
	T(x) := \{v\in S : \phi'_v(x) \neq 0\}.
	$$ 
	Assume that for all but finitely many $x \in [0, 1]$ 
	we have
	\eq{}{\#T(x)_\R + 2 \prt{\# T(x)_\CO} > \frac{d}{2}.} 
	Then $\phi^{-1} \prt{\BA_{K}} $ is winning, 
	and hence $\BA_{K}\cap\, \phi([0,1])$ is thick in $ \phi([0,1])$.
\end{Thm}

In particular it follows that the set $\BA_{K}$ is itself thick. 
The goal of this  paper is to prove a stronger property of the set $\BA_{K}$, from which the conclusions of Theorem \ref{Thm:ESK} and Theorem \ref{Thm:EGL} follow. Namely, following \cite{McMullen10, BroderickFishmanKleinbockReichWeiss12,  FishmanSimmonsUrbanski13}  in \S\ref{Sec:Schmidt game}  we describe so-called \emph{$\Hc$-absolute game} on a subset of a complete metric space, where $\Hc$ is a collection of subsets of the space. In this game, compared with McMullen's version, Alice deletes neighborhoods of sets from $\Hc$ instead of points.  Winning sets of this game are also winning for the regular Schmidt's game, and, moreover, the winning properties are inherited by certain nice subsets
of the ambient space. 

To state our main result we need to introduce some notation. If $T$ is a subset of $S$ and $\x\in K_{S}$, we let 
\eq{hp}{\hplane{\x}{T} = \cbrk{ \y \in K_{S}: y_v = x_v\   \forall\, v \in T }}
be the affine subspace of $K_{S}$ passing through $\x$ and orthogonal to coordinate directions corresponding to $v \in T$. Then let $\Hc_{K}$ be the following family of affine subspaces:
\eq{Hc}{\Hc_{K} = \cbrk{\hplane{\x}{T} : \x \in \iv[S]{K}, T \subseteq S \text{ with } \#T_\R + 2 \prt{\#T_\CO} > \frac{d}{2}}.}
Our choice of $\Hc_{K}$ will be explained in~\S\ref{Sec:Proof}.  
In particular, we will show in Proposition~\ref{Prop:Optimal} that for any $\mathcal{L} \in \Hc_K$,
\eq{}{\mathcal{L} \cap \BA_K =\varnothing,}
so the collection $\Hc_K$ is, in some sense, optimal.

Here is our main theorem:

\begin{Thm}
	\lab{Thm:Main} $\BA_{K}$ is $\Hc_{K} $-absolute winning.
\end{Thm}

As a consequence of the main theorem,  at the end of~\S\ref{Subsec:H-absolute game} we will prove the following corollary, which will imply Theorem~\ref{Thm:EGL}.

\begin{Cor} \lab{Cor:EGL}
	Let $\phi: [0, 1] \to K_S$ be a $C^1$ curve such that for all but countably many $x \in [0, 1]$, we have
	\eq{}{\#T(x)_\R + 2 \prt{\# T(x)_\CO} > \frac{d}{2}.} 
	Then $\phi^{-1}(\BA_K)$ is absolute winning.
\end{Cor}

Moreover, notice that for any $\Lc = \hplane{\x}{T}\in \Hc_{K} $,
\eq{HDim}{\dim_{\R} (\Lc ) = d - (\# T_\R) - 2(\# T_\CO) = (\# \prt{S \smallsetminus T}_\R) + 2(\# \prt{S \smallsetminus T}_\CO) < \frac{d}{2}.}
So when $d = 1, 2$, or when $K$ is totally complex and $d = 4$,   $\Hc_{K}$ consists of only points $(0$-dimensional subspaces), and hence the $\Hc_{K}$-absolute game coincides with the absolute game defined in \cite{McMullen10}. Thus we have the following extension of McMullen's result:

\begin{Cor} \lab{Cor:AW}
	If $d = 1, 2$, or $K$ is totally complex quartic number field, then $\BA_{K}$ is absolute winning.
\end{Cor}
In particular, this implies that in those cases $\phi^{-1} \prt{\BA_{K}} $ is winning for any smooth curve $\phi : [0, 1] \to K_{S}$, regardless of directions in which its derivative is zero.  Moreover, in view of the inheritance property 
(see \S\ref{Subsec:H-absolute game} for more details), we obtain the following corollary strengthening Theorem~\ref{Thm:ESK}:

\begin{Cor} \lab{Cor:ESK}
	Let $K$ be an imaginary quadratic number field, and $Y \subseteq \CO \cong K_S$ be the support of an Ahlfors regular measure, then $\BA_K$ is absolute winning on $Y$.
\end{Cor}


The structure of this paper is as follows.  In \S\ref{Sec:Number field}  we discuss the setting of \da\ in number fields  in more details.  Our main tools, the height function and Dani's correspondence (Proposition~\ref{Prop:Dani Corr}), are described in \S\ref{Subsec:BA}.  In \S\ref{Sec:Diffuseness} we define $\Hc$-diffuse sets, and  in \S\ref{Sec:Schmidt game}  we  introduce Schmidt game and the $\Hc$-absolute game and derive Corollaries \ref{Cor:EGL} and \ref{Cor:ESK} from the main theorem.  The proof of Theorem \ref{Thm:Main} is given in \S\ref{Sec:Proof}.  In ~\hyperref[Sec:Appx A]{Appendix A}, we will show that the badly approximable $S$-numbers considered by Hattori~\cite{Hattori07} are in fact the same as $\BA_K$.  And finally, in~\hyperref[Sec:Appx B]{Appendix B} we give the proofs of basic properties of $\Hc$-absolute winning  sets.

\subsection*{Acknowledgements}  We thank S.G.~Dani, M.~Einsiedler, L.~Fishman, A.~Ghosh, E.~Lindenstrauss, K.~Merrill and G.~Tomanov for helpful discussions, and the MSRI for its hospitality during Spring 2015.


\section{Diophantine approximation 
in number fields} \lab{Sec:Number field}
\subsection{Dirichlet Theorem}
\lab{Subsec:DT}
As in classical theory of Diophantine approximation, the justification of our definition of badly approximable $S$-numbers~\eqr{BA in K} is the following version of Dirichlet's Theorem:

\begin{Prop}[Strong Dirichlet Theorem]
	\lab{Prop:Strong Dirichlet}
	There exists a constant $C = C_K > 0$ depending only on $K$, such that for every $\x \in K_{S}$ and for every $Q > 0$, there exists $p \in \BO$, $q \in \BO \smallsetminus \cbrk{0}$ with 
	\eq{Strong Dirichlet}{\norm{\iv[S]{q} \cdot \x + \iv[S]{p}} \leq C Q^{-1} \text{\quad and \quad} \norm{\iv[S]{q}} \leq Q.}
\end{Prop}

For more notations from algebraic number theory, let us denote the \emph{local degree} at $v \in S$ by:
\eq{}{d_v := [K_v : \Q_v] = [K_v : \R] = \begin{cases}	1	&, v \in S_\R \\ 2	&, v \in S_\CO \end{cases}.}
And for $a \in \BO$, (the absolute value of) the field norm of $a$ is defined by:
\eq{}{|N(a)| = \prod_{v \in S} |\iv[v]{a}|^{d_v}.}

To prove Proposition~\ref{Prop:Strong Dirichlet}, we will make use of the following (special case of) lemma of Burger~\cite{Burger92}:
	
\begin{Lem}[{\cite[Lemma 5.1]{Burger92}}] \lab{Lem:Burger}
	Let $\x \in K_S$ and for each $v \in S$, let $0 < \varepsilon_v < 1 \leq \delta_v$ such that
	\eq{}{\prod_{v \in S} \prt{\delta_v \varepsilon_v}^{d_v} \geq \prt{\frac{2}{\pi}}^{2(\#S_\CO)} \abs{D_K},}
	where $D_K$ is the discriminant of $K$.
	Then there exists $p, q \in \BO$, $q \neq 0$, satisfying
	\eq{}{\abs{\iv[v]{q} x_v + \iv[v]{p}} \leq \varepsilon_v \text{\quad and \quad} \abs{\iv[v]{q}} \leq \delta_v, \text{\quad for
	 all } v \in S.}
\end{Lem}

\begin{proof}[Proof of Proposition~\ref{Prop:Strong Dirichlet}]
	Let
	\eq{}{C = \prt{\prt{\frac{2}{\pi}}^{2\prt{\# S_\CO}} \abs{D_K}}^{1/d},}
	then by applying Lemma~\ref{Lem:Burger} with
	\eq{}{\delta_v = Q \text{\quad and \quad} \varepsilon_v = CQ^{-1} \text{\quad for all } v \in S,}
	we can find $p \in \BO$, $q \in \BO \smallsetminus \cbrk{0}$ satisfying~\eqr{Strong Dirichlet}.
\end{proof}

\begin{Thm}[Weak Dirichlet Theorem]
	\lab{Thm:Weak Dirichlet}
	There is a constant $C = C_K > 0$ depending only on $K$, such that for every $\x \in K_{S}$, there are infinitely many $p, q \in \BO, q \neq 0$ satisfying:
	\eq{DT}{\norm{\iv[S]{q} \cdot \x + \iv[S]{p}} \leq C \norm{\iv[v]{q}}^{-1}. }
\end{Thm}

\begin{proof} 
	Note that
	\eq{}{\norm{\iv[S]{q} \cdot \x + \iv[S]{p}} = 0 \iff \x = \iv[S]{\frac{p}{q}}.}
	
	If $\x \in \iv[S]{K}$, then the existence of infinitely many $p, q$ is obvious. 
	Otherwise, for every $p, q \in \BO$ with $q \neq 0$,
	\eq{}{\norm{\iv[S]{q} \cdot \x + \iv[S]{p}} > 0.}
	Moreover, for every $Q > 0$,
	\eq{}{\# \cbrk{q \in \BO: \norm{\iv[S]{q}} \leq Q} < \infty.}
	Hence, by letting $Q \to \infty$ and applying Proposition~\ref{Prop:Strong Dirichlet}, we can find infinitely many $p, q \in \BO$, $q \neq 0,$ such that
	\eq{}{\norm{\iv[S]{q} \cdot \x + \iv[S]{p}} \leq C Q^{-1} \leq C \norm{\iv[S]{q}}^{-1}.}
\end{proof}

Note that $\x \in \BA_K$ 
if and only if the constant $C$ in~\eqr{DT} cannot be replaced by an arbitrarily small constant.  In other words, elements of $\BA_K$ are the witnesses of the optimality of Theorem~\ref{Thm:Weak Dirichlet}.

\begin{Rem}
	\mbox{ }
	\begin{itemize}
	\item[(i)]	When $\#S = 1$, $\norm{\iv[S]{q}}$ can be factored out from the left hand side of~\eqr{DT}:
	\eq{}{\norm{\iv[S]{q} \cdot \x + \iv[S]{p}} = \norm{\iv[S]{q} \cdot \prt{\x + \iv[S]{\frac{p}{q}}}} = \norm{\iv[S]{q}} \cdot \norm{\x + \iv[S]{\frac{p}{q}}},}
	and that implies \eqr{DT} is equivalent to:
	\eq{DT'}{\norm{\x + \iv[S]{\frac{p}{q}}} \leq C \norm{\iv[S]{q}}^{-2}.}
	In particular for this case, $\x \in \BA_K$ if and only if there exists $c > 0$ such that for all $p \in \BO$, $q \in \BO \smallsetminus \cbrk{0}$, 
	\eq{BA'}{\norm{\x + \iv[S]{\frac{p}{q}}} > c \norm{\iv[S]{q}}^{-2}.}
	So for $K = \Q$, \eqr{BA'} is exactly \eqr{BA in R}; and 
	for $K$ as in Theorem~\ref{Thm:ESK}.
	this is the definition used in \cite{EsdahlKristensen10}.
	
	\item[(ii)]	In~\cite{Hattori07}, \eqr{BA'} is used to define badly approximable numbers when $K$ is real quadratic or totally complex quartic.  In fact, it can be shown that in this case~\eqr{BA'} is equivalent to~\eqr{BA in K}.
	\end{itemize}
\end{Rem}


\subsection{Dani Correspondence and the height function } \lab{Subsec:BA}
By abusing 
notation, we let the diagonal embedding $\iv[S]{}: K \to K_S$ be extended to matrices $\iv[S]{} : M_{m, n} (K) \to M_{m, n} (K_S)$ by
\eq{}{\prt{\iv[S]{A}}_{i, j} = \iv[S]{A_{i, j}} \text{\quad for } A = (A_{i, j}) \in M_{m, n}(K).}

We also extend the sup norm in $K_S$ to $K_S^2$:
\eq{}{\norm{\vect{z}} = \max \cbrk{\norm{\z_1}, \norm{\z_2}} \text{\quad where } \vect{z} = \pvect{\z_1}{\z_2} \in K_S^2.}

Also note that a $K_S$-vector $\vect{z} \in K_S^2$ can be viewed in two ways:
\eq{}{\vect{z} = \pvect{\mathbf{z}_1}{\mathbf{z}_2} \in K_S^2 \text{ \quad or \quad} \vect{z} = \prt{\vec{z_v}}_{v \in S} \in \prod_{v \in S} K_v^2.}

By abusing 
 notation again, we also use $\norm{\cdot}$ for the sup norm in $K_v^2$.  In particular,
\eq{}{\norm{\vect{z}} = \max \cbrk{\norm{\z_1}, \norm{\z_2}} = \max_{v \in S} \norm{\vec{z}_v}.}

{We let $K$ act naturally on $K_S^2$ by:
\eq{action}{a \vect{z} = \pvect{\iv[S]{a} \cdot \z_1}{\iv[S]{a} \cdot \z_2}, \text{ \quad for } a \in K, \vect{z} \in K_S^2.}

Let $G = \SL_2 (K_S) \cong \prod_{v \in S} \SL_2(K_v)$, then $\Gamma = \iv[S]{\SL_2(\BO)}$ is a lattice in $G$, and we denote the homogeneous space $G/\Gamma$  by $X_K$.  The isomorphism $K_S \cong \R^d$ induces an embedding $\SL_2(K_S) \hookrightarrow \SL_{2d}(\R)$.  Hence, $X_K$ can be identified with a proper subset of the space $\SL_{2d}(\R)/\SL_{2d}(\Z)$ of unimodular lattices  in $\R^{2d}$.   {Via the map: $g\Gamma \mapsto \iv[S]{g}\BO^2$}, an element $\Lambda \in X_K$ can be identified  
with a discrete free rank $2$ $\BO$-module of $K_S^2$ having a basis $\cbrk{\vect{x}, \vect{y}}$ 
such that for every $v \in S$, $\cbrk{ \vec{x}_v, \vec{y}_v}$ forms a parallelepiped of area 1 in $K_v^2$ (see Section 2 of~\cite{EinsiedlerGhoshLytle13} for more details).

 Following the ideas of Dani, the space $X_K$ was used in {\cite{EinsiedlerGhoshLytle13}} to give an alternative description of badly approximable $S$-numbers. 
Let us associate each $\x \in K_{S}$ with the lattice $\Lambda_\x = u_\x \Gamma \in X_K$, where $u_\x \in G$ is defined by:
\eq{Ax}{(u_\x)_v = \begin{pmatrix} 1 & x_v \\  & 1 \end{pmatrix}.}
Also, for each $t \in \R$, we define $g_t \in G$ by:
\eq{gt}{(g_t)_v = \begin{pmatrix} e^{t} & \\ & e^{-t} \end{pmatrix}.}
It is shown in \cite[Proposition 3.1]{EinsiedlerGhoshLytle13} that 
\eq{danicorr}
{\x \in \BA_K \iff\text{ the orbit  $\cbrk{g_t \Lambda_\x: t \geq 0}$ is bounded in }X_K.}
In $K_S^2$, there is a height function more suitable for our needs than $\norm{\cdot}$, which can be thought as a natural extension of the field norm in $K$.  For a vector $\vect{z} = \pvect{\z_1}{\z_2} \in K_{S}^2$, we define the \emph{height} of $\vect{z}$ to be:
{\eq{Height}{\height{\vect{z}} = \height{\vvect{\z_1}{\z_2}} := \prod_{v \in S} \max \cbrk{ \abs{(\z_1)_v}, \abs{(\z_2)_v}}^{d_v} = \prod_{v \in S} \norm{\vec{z_v}}^{d_v}. }}
{We also use the same notation $\height{\cdot}$ to define the height on $K_S$ similar to~\eqr{Height}:
\eq{}{\height{\x} := \prod_{v \in S} \abs{x_v}^{d_v} \text{ \quad for } \x \in K_S.}}
In particular, if $a \in K$ then
\eq{}{\abs{N(a)} = \height{\iv[S]{a}}.}

The following lemma gives an important property of the height function under the action of the group of units $\BO^\times$:

\begin{Lem}[{\cite[Lemma 2.4]{EinsiedlerGhoshLytle13
}}] \lab{Lem:Height properties}
	There exists a constant $C \geq 1$ depending only on $K$ such that if $\height{\vect{z}} \neq 0$ then there exists a unit $\xi \in \BO^\times$ satisfying:
	\eq{H vs abs}{C^{-1} \height{\vect{z}}^{\frac{1}{d}} \leq \norm{\iv[v]{\xi} \vec{z_v}} \leq C \height{\vect{z}}^{\frac{1}{d}} \text{\quad for all } v \in S.}
\end{Lem}

	Note that just as in~\cite[Lemma 5.10]{KleinbockTomanov03}, this relation holds for higher dimensions and a more general $S$ with a suitable extension of the height function.

This height function provides 
a way to measure the size of a lattice in $K_S^2$.  Namely, let us define
\eq{DeltaH}{\delta_H \prt{\Lambda} = \min \cbrk{\height{\vect{z}} : \vect{z} \in \Lambda \smallsetminus \cbrk{0}}.}

Using the above lemma, it is shown in \cite[Proposition 2.5]{EinsiedlerGhoshLytle13} that 
	 a subset $A \subseteq X_K $  is relatively compact if and only if $\inf \cbrk{\delta_H \prt{\Lambda} : \Lambda \in A} > 0$. Combining it with \eqr{danicorr}, we arrive at

\begin{Prop}[{\cite{EinsiedlerGhoshLytle13}}] \lab{Prop:Dani Corr}
	\eq{}{\x \in \BA_K \iff \inf \cbrk{\delta_H \prt{g_t \Lambda_\x}: t \geq 0} > 0.}
\end{Prop}



\begin{Rem}\label{erg}
	Applying Moore's Ergodic Theorem and arguing similar to~\cite[Section 2]{Dani85} or~\cite[Theorem 8.7]{KleinbockMargulis99}, it can be deduced from Proposition~\ref{Prop:Dani Corr} 
	that $\BA_K$ has Lebesgue measure zero.
\end{Rem}

\section{\texorpdfstring{$\Hc$}{H}-diffuse sets} \lab{Sec:Diffuseness}
Before discussing the modification of Schmidt's game needed for our purposes, in this section we survey the notion of \textsl{diffuseness} introduced in \cite{BroderickFishmanKleinbockReichWeiss12} in order to describe sets which can serve as nice playgrounds for those games. 

Let $(X, \dist{})$ be a complete metric space. Each pair $B = (x, \rho)$, where $x \in X$ and $\rho > 0$, is called a \emph{formal ball in $X$}, and we denote corresponding closed ball in $X$ by:
\eq{}{\B(B) = \B(x, \rho) := \cbrk{y \in X: \dist{x, y} \leq \rho}.}
More generally, if $\Lc$ is a subset of $X$ and $\rho > 0$, we denote by  $\Lc^{\prt{\rho}}$  the \emph{$\rho$-neighborhood of $\Lc$:}
\eq{}{\Lc^{\prt{\rho}} := \cbrk{x \in X: \dist{x, \Lc} \leq \rho}.}

The projection maps from $X\!\times\!\R_{>0}$ to the first and second factors are denoted by $c$ and $r$ respectively:
\eq{}{c(x, \rho) = x \text{ and } r(x, \rho) = \rho \text{\quad for } (x, \rho) \in X\!\times\!\R_{>0}.}

We define a partial order $\preceq$ on the set $X\!\times\R_{>0}$ of formal balls:
\eq{}{B_1 \preceq B_2 \iff \dist{c(B_1), c(B_2)} \leq r(B_2) - r(B_1).}
In particular, $B_1 \preceq B_2$ implies $\B(B_1) \subseteq \B(B_2)$, but the converse might not hold in general. 
Notice that if $X$ is a real Banach space, then $\B(B)$ is uniquely determined by $B$, but in general, there might exist $B_1 \neq B_2$ such that $\B(B_1) = \B(B_2)$.  For instance, we can easily find such $B$'s when $X$ is the Cantor set with the induced metric from $\R$.

Now let $\Hc$ be a non-empty collection of closed subsets of $X$, and pick $0 < \beta < 1$. Following 
\cite{BroderickFishmanKleinbockReichWeiss12}, we say that a subset $Y \subseteq X$ is \emph{$(\Hc, \beta)$-diffuse} if for every $x \in Y$, there exists $\rho_x > 0$ such that for any $(y, \rho) \in Y\!\times\!\R_{>0}$ with $(y, \rho) \preceq (x, \rho_x)$, and for any $\Lc \in \Hc$:
\eq{diffuse}{Y \cap \prt{\B(y, \rho) \smallsetminus \Lc^{(\beta \rho)}} \neq \varnothing.}
Note that $(\Hc, \beta)$-diffuseness implies $(\Hc, \beta')$-diffuseness for any $0 < \beta' \leq \beta$.  $Y$ is said to be \emph{$\Hc$-diffuse} if it is $(\Hc, \beta)$-diffuse for some $0 < \beta < 1$.  

\begin{Rem}
	\mbox{ }
	\begin{itemize}
	\item[(i)]	Our definition of diffuse sets is slightly weaker than the definition used in~\cite{BroderickFishmanKleinbockReichWeiss12}, in that they required $\rho_x$ to be 
	uniformly bounded 
	below for all  $x\in Y$.
	
	\item[(ii)]	The constant $\rho_x$ at first seems to depend also on $\beta$, but notice that if $\rho_x$ works of $\beta$, it also works for any $0 < \beta' \leq \beta$.
	\end{itemize}  
\end{Rem}

In the set-up of \cite{BroderickFishmanKleinbockReichWeiss12} $X = \R^d$ with the Euclidean metric, and  the collection $\Hc$ consisted of all $k$-dimensional affine subspaces of $\R^d$, and the corresponding property was called \emph{$k$-dimensional diffuseness}. Many examples of $k$-dimensionally diffuse sets were exhibited there. For example, it is clear that any $m$-dimensional smooth submanifold of $\R^d$ is $k$-dimensionally diffuse whenever $m > k$. Also if   $\mu$  is an absolutely decaying measure on $\R^d$ (see \cite{KleinbockLindenstraussWeiss04, PollingtonVelani05} for definition) 
then, as shown in   \cite[Proposition 5.1]{BroderickFishmanKleinbockReichWeiss12} $ \supp\,\mu$ is  $k$-dimensionally diffuse for all $1 \le k < d$).

The following two  examples of diffuse sets will be used in the proofs of Corollaries~\ref{Cor:EGL} 
and 
\ref{Cor:ESK} 
in the next section:

\begin{Eg} \lab{Eg:Ahlfors}
	For an arbitrary metric space $X$, let $\Hc = \cbrk{\cbrk{x}: x \in X}$  and let $Y = \supp \mu$ be the support of an Ahlfors regular measure $\mu$ on $X$ (see~\eqr{Ahlfors} for the definition); then $Y$ is $\Hc$-diffuse.  To see this, for every $0 < \beta < \prt{\frac{a}{b}}^{1/\delta}$, $(y, \rho) \in X\!\times\!\R_{> 0}$ with $\rho \leq r_0$, and $x \in X$, write
	\eq{}{\mu \prt{\B(y, \rho) \smallsetminus \B(x, \beta \rho)} \geq a \rho^\delta - b (\beta \rho)^\delta > 0.}
	In particular,
	${Y \cap \prt{\B(y, \rho) \smallsetminus \B(x, \beta \rho)} \neq \varnothing}$; that is, \eqr{diffuse} holds.
\end{Eg}

\begin{Prop} \lab{Prop:H-diffuse curve}
	Let $\mathcal{S}$ be a closed set of linear subspaces in $\R^n$, and let
	\eq{}{\Hc = \cbrk{\Lc + \vec{x}: \Lc \in \mathcal{S}, \vec{x} \in \R^n}.}
	If $\phi: [0, 1] \to \R^n$ is an injective $C^1$ map such that the curve $\phi([0, 1])$ is not tangent to any affine subspace in $\Hc$,  then $\phi([0, 1])$ is $\Hc$-diffuse.
\end{Prop}

\begin{proof}
	For every $\Lc \in \mathcal{S}$, let $\pi_{\Lc^\perp}: \R^n \to \Lc^\perp$ be the projection onto its orthogonal complement.  Let
	\eq{}{a = \min_{\substack{0 \leq t \leq 1 \\ \Lc \in \mathcal{S}}} \frac{\norm{\prt{\pi_{\Lc^\perp} \circ \phi}'(t)}}{\norm{\phi'(t)}} \text{ \quad, \quad} b = \max_{0 \leq t \leq 1} \norm{\phi'(t)} \text{ \quad and \quad}  {c = \min_{0 \leq t \leq 1} \norm{\phi'(t)}}.}
Note that $a > 0$ by the non-tangency and compactness of $\mathcal{S}$.
In particular, for every $0 \leq t \leq 1$ and $\Lc \in \mathcal{S}$:
	\eq{}{\norm{\prt{\pi_{\Lc^\perp} \circ \phi}'(t)} \geq  {ac}.}
	
	For every $\vec{x} \in \phi([0, 1])$, we denote $t_{\vec{x}} = \phi^{-1} \prt{\vec{x}}$.  For an arbitrary $\vec{x} \in \phi([0, 1))$ (the proof is similar at the other endpoint), we let $\rho_{\vec{x}} > 0$ be sufficiently small such that 
	\eq{}{\partial \B(\vec{x}, \rho_{\vec{x}}) \cap \phi((t_{\vec{x}}, 1)) \neq \varnothing;}
	and	for $\vec{y}, \vec{z} \in \B(\vec{x}, \rho_{\vec{x}})$ and for every $\Lc \in \mathcal{S}$:
	\eq{}{\frac{\norm{\pi_{\Lc^\perp} \prt{\vec{y} - \vec{z}} }}{\abs{t_{\vec{y}} - t_{\vec{z}}}} \geq  {\frac{ac}{2}}.}
	
	Let $0 < \beta < {\frac{ac}{4b \sqrt{n}}}$, and $(\vec{y}, \rho) \in \phi([0, 1])\!\times\!\R_{>0}$ be arbitrary with $(\vec{y}, \rho) \preceq (\vec{x}, \rho_{\vec{x}})$.  Let
	\eq{}{t_0 = \min \phi^{-1} \prt{\partial \B(\vec{y}, \rho) \cap [t_{\vec{y}}, 1]} = \min \cbrk{t > t_{\vec{y}} : \norm{\phi(t) - \vec{y}} = \rho}.}
	
	Let $\Lc + \vec{z} \in \Hc$ be arbitrary.  For any $0 \leq t_1 < t_2 \leq 1$ such that $\phi(t_1), \phi(t_2) \in \B(\vec{y}, \rho) \cap \prt{\Lc + \vec{z}}^{\prt{\beta \rho}}$, the arclength of $\phi([t_1, t_2])$ is:
	\begin{align*}
	l\big(\phi([t_1, t_2])\big) &= \int_{t_1}^{t_2} \sqrt{\abs{\phi'_1(t)}^2 + ... + \abs{\phi'_n(t)}^2} dt \\
	&\leq \abs{t_2 - t_1} \cdot b\sqrt{n} \\
	&\leq \frac{2}{ac} \norm{\pi_{\Lc^\perp} \prt{ \phi (t_2) - \phi (t_1)}} \cdot b \sqrt{n} \\
	&\leq \frac{2b}{ac} \cdot 2\beta \rho \cdot \sqrt{n} 
	< \rho \leq l\big(\phi([t_{\vec{y}}, t_0])\big).
	\end{align*}
	In particular, for any $\Lc + \vec{z} \in \Hc$:
	\eq{}{\phi([0, 1]) \cap \prt{\B(\vec{y}, \rho) \smallsetminus \prt{\Lc + \vec{z}}^{\prt{\beta \rho}}} \neq \varnothing.}
	This shows that $\phi([0, 1])$ is $\Hc$-diffuse.
\end{proof}

We close the section with the following useful property of diffuse sets:

\begin{Lem}
\lab{Lem:diffuse}
	If $Y$ is $(\Hc, \beta)$-diffuse, then for any $0 < \beta' \leq \frac{\beta}{2 + \beta}$, $x \in Y$, $(y, \rho) \in Y\!\times\R_{>0}$ with $(y, \rho) \preceq (x, \rho_x)$, and any $\Lc \in \Hc$, there exists $z \in Y \cap \B(y, \rho)$ with
	\eq{diffresult}{(z, \beta' \rho) \preceq (y, \rho) \text{~\quad and \quad} \dist{z, \Lc} > 2\beta'\rho.}
\end{Lem}

\begin{proof}
	The proof is identical to the proof of~\cite[Lemma 4.3]{BroderickFishmanKleinbockReichWeiss12} which stated the same result in the $k$-dimensional diffuseness set-up.  It is sufficient to prove the lemma for $\beta' = \frac{\beta}{2 + \beta} < \beta$.  Since $(y, (1 - \beta')\rho) \preceq (y, \rho) \preceq (x, \rho_x)$, there exists
	\eq{}{z \in Y \cap \prt{\B\prt{y, (1 - \beta')\rho} \smallsetminus \Lc^{\prt{\beta(1 - \beta')\rho}}} = Y \cap \prt{\B\prt{y, (1 - \beta')\rho} \smallsetminus \Lc^{\prt{2\beta'\rho}}}.}
	Thus
	\eqr{diffresult} holds.
	\end{proof}
	}

\section{Schmidt game and its variants} \lab{Sec:Schmidt game}
\subsection{Schmidt's $(\alpha, \beta)$-game} \lab{Subsec:Schmidt game}
In~\cite{Schmidt66}, W. Schmidt introduced an infinite game between two players, called Alice and Bob, which has been shown to be a powerful tool in Diophantine approximation.  The setup of Schmidt's game requires the followings:
\begin{enumerate}
	\item	The playground is a complete metric space $(X, \dist{})$.
	\item	Two real numbers $\alpha, \beta$, with $0 < \alpha, \beta < 1$, are parameters associated with Alice and Bob respectively.
	\item	A subset $W \subseteq X$ is Alice's \emph{target set}.
\end{enumerate}

The game proceeds by Alice and Bob alternatively picking $B_1, A_1, B_2, A_2, ...$ with $A_n, B_n \in X\!\times\!\R_{>0}$ satisfying:
\begin{enumerate}[label=(S\arabic*)]
	\item	\lab{S1} $B_n$'s are Bob's choices, $A_n$'s are Alice choices.
	\item	\lab{S2} $r(A_n) = \alpha r(B_n)$ and $A_n \preceq B_n$ for all $n \geq 1$.
	\item	\lab{S3} $r(B_n) = \beta r(A_{n - 1})$ and $B_{n + 1} \preceq A_n$ for all $n \geq 2$.
\end{enumerate}
In particular, the result of a play forms a nested sequence of closed balls: 
\eq{}{\B(B_1) \supseteq \B(A_1) \supseteq \B(B_2) \supseteq \B(A_2) \supseteq \dots}

Since $X$ is complete, and the radii $\lim_{n \to \infty} r(A_n) = \lim_{n \to \infty} r(B_n) = 0$, their intersection is a point, denoted by $x_\infty$:
\eq{}{\bigcap_{n = 1}^\infty \B(A_n) = \bigcap_{n = 1}^\infty \B(B_n) = \cbrk{x_\infty}.}
If Alice has a strategy that guarantees $x_\infty \in W$ regardless of what Bob does, then we say that $W$ is \emph{\abwin}.  If $W$ is \abwin~for every $0 < \beta < 1$, then it is called \emph{\awin}.  And finally, $W$ is called \emph{winning}~if it is \awin~for some $0 < \alpha < 1$.

Schmidt showed that winning sets have remarkable properties:

\begin{Lem}[{\cite{Schmidt66}}]
	\mbox{ }
	\begin{itemize}
		\item[\rm (i)]		If $X = \R^n$ then winning sets are thick.
		\item[\rm (ii)]		A countable intersection of \awin~sets is again \awin.
		\item[\rm (iii)]	If $W$ is \awin~ and $\phi: X \to X$ is bi-Lipschitz, then $\phi(W)$ is $\alpha'$-winning, with $\alpha'$ depends on $\alpha$ and the bi-Lipschitz constant of $\phi$.
	\end{itemize}
\end{Lem}



\subsection{$\Hc$-absolute game on $\Hc$-diffuse sets} \lab{Subsec:H-absolute game}
Generalizing the ideas of the absolute game of McMullen~\cite{McMullen10} and the $k$-dimensional absolute game of Broderick et  al.~\cite{BroderickFishmanKleinbockReichWeiss12}, Fishman, Simmons and Urbanski~\cite{FishmanSimmonsUrbanski13} introduced the \emph{$\Hc$-absolute game} which we will describe as follows.

Let $(X, \dist{})$ be a complete metric space,  let $\Hc$ be a non-empty collection of closed subsets of $X$, and pick $0 < \beta < 1$.  For a non-empty closed subset $Y \subseteq X$, 
Alice and Bob play the \emph{$\Hc$-absolute game on $Y$} by alternatively choosing an infinite sequence $B_1, A_1, B_2, A_2, ...$ satisfying the rules below:

\begin{enumerate}[label=(H\arabic*)]
	\item	\lab{H1} Bob chooses $B_n = (c_n, r_n) \in Y\!\times\!\R_{>0}$.
	\item	\lab{H2} Alice chooses $A_n = (\Lc_n, \rho_n) \in \Hc\!\times\!\R_{>0}$ with $0 < \rho_n \leq \beta r_n$.
	\item	\lab{H3} $\beta r_n \leq r_{n + 1} \leq r_n$ and $B_{n + 1} \preceq B_n$.
	\item	\lab{H4} $\dist{c_{n + 1}, \Lc_n} > \rho_n + r_{n + 1}$.
	\item	\lab{H5} If at some point of the game, Bob has no choices $B_{n + 1} \in X\!\times\!\R_{>0}$ satisfying~\ref{H3} and~\ref{H4}, then Bob wins, and the game is terminated.
\end{enumerate}

In particular,   conditions~\ref{H3} and~\ref{H4} imply that $\B(B_{n + 1}) \subseteq \B(B_n) \smallsetminus \Lc_n^{\prt{\rho_n}}$.
Hence, we can think of Alice's move as deleting a neighborhood of a closed subset $\Lc$ in $\Hc$.  

\begin{Rem}
	As a convention, the only neighborhood of the empty set is the empty set itself, and the distance from any point of $X$ to the empty set is infinite.  Hence, the empty set, if in $\Hc$, can be considered as a `dummy move' for Alice.
\end{Rem}

A set $W \subseteq X$ is said to be \emph{$(\Hc, \beta)$-absolute winning on $Y$} if Alice has a strategy to ensure that at every stage of the game, Bob always has at least one choice, and:
\eq{}{W \cap Y \cap \bigcap_{n = 1}^\infty \B(B_n) \neq \varnothing}
regardless of Bob's strategy.  We will say that $W$ is \emph{$\Hc$-absolute winning on $Y$} if there exists $\beta_0 > 0$ such that $W$ is $(\Hc, \beta)$-absolute winning on $Y$ for all  $0 < \beta < \beta_0$.

\begin{Rem}
	When $Y = X$ we will drop `on $Y$' and simply say that $W$ is $(\Hc, \beta)$-absolute winning and $\Hc$-absolute winning accordingly.
\end{Rem}

\begin{Rem}
	It is clear from the rules of the game that for any collection $\Hc$, if $W \subseteq X$ is $\Hc$-absolute winning on $Y$, then $W \cap Y$ is dense in $Y$ with the subspace topology, since that implies $W \cap Y \cap \B(B_1) \neq \varnothing$ for arbitrary $B_1$.
\end{Rem}

\begin{Eg}
	\mbox{}
	\begin{itemize}
		\item[(i)]	When $\Hc = \cbrk{ \cbrk{x}: x \in X}$ is the set of singletons in $X$, the $\Hc$-absolute game is the absolute game considered by McMullen~\cite{McMullen10}.
		\item[(ii)]	When $X = \R^d$ and $\Hc$ is the collection of $k$-dimensional affine subspaces in $\R^d$, we get the $k$-dimensional absolute game of Broderick et  al~\cite{BroderickFishmanKleinbockReichWeiss12}.  See also \cite{BroderickFishmanSimmons13, NesharimSimmons14, KleinbockWeiss15, AnGuanKleinbock15} for other appearances of games of this type.
	\end{itemize}
\end{Eg}

The degenerate case when Alice's moves leave Bob without any legitimate choice at some point of the game can be avoided when $X = Y = \R^n$ by restricting $0 < \beta < 1/3$ as in~\cite{McMullen10, BroderickFishmanKleinbockReichWeiss12}.  For the $k$-dimensional absolute game playing on a subset $Y$ of $\R^n$, Broderick et al.~\cite{BroderickFishmanKleinbockReichWeiss12} showed that the  {$k$-dimensional diffuseness} of $Y$
is a sufficient condition for  the game to last infinitely. More generally,  
Lemma \ref{Lem:diffuse}, essentially taken  from~\cite{BroderickFishmanKleinbockReichWeiss12}, makes sure that Bob always has legitimate moves in the $\Hc$-absolute game played on $Y$ if $Y$ is $\Hc$-diffuse, allowing us to ignore condition~\ref{H5}:

\begin{Rem} \lab{Rem:radius 0}
	By Lemma~\ref{Lem:diffuse}, to show that $W$ is $\Hc$-absolute winning on an $\Hc$-diffuse set $Y$, it suffices to assume that $r(B_n) \to 0$.  So by rearranging the indices, it suffices to assume that $\rho_{c(B_1)} = r(B_1)$.  So if $B_1 = (x, \rho) \in Y\!\times\R_{>0}$ be Bob's arbitrary first move, then $r(B_{N}) < \rho_x$ for all $N$ sufficiently large.  By rearranging the indices, we can assume that the $\rho_{c(B_1)} = \rho_x = r(B_1)$.
\end{Rem}

We now list various properties of   $\Hc$-absolute winning sets. To make the exposition more streamlined, their proofs are postponed until~\hyperref[Sec:Appx B]{Appendix B}.

The next lemma shows that a set $\Hc$-absolute winning on an $\Hc$-diffuse set will be absolute winning for all reasonable $\beta$'s:

\begin{Lem} \lab{Lem:Extend beta}
	If $W$ is $\Hc$-absolute winning on $Y$ and $Y$ is $(\Hc, \beta)$-diffuse, then $W$ is $(\Hc, \beta')$-absolute winning on $Y$ for all $0 < \beta' \leq \frac{\beta}{2 + \beta} $.
\end{Lem}

\begin{proof}
	See~\S\ref{Appx:Proof of Extend beta}.
\end{proof}

Using the above lemma, we can derive that $\Hc$-absolute winning on $\Hc$-diffuse set implies winning in Schmidt's sense:

\begin{Prop} \lab{Prop:winning} 	
	If $W$ is $\Hc$-absolute winning on $Y$, and $Y$ is $(\Hc, \beta)$-diffuse, then $W \cap Y$ is $\frac{\beta}{2 + \beta}$-winning when we play Schmidt game on $Y$ equipped with the induced metric.
\end{Prop}

\begin{proof}
	See~\S\ref{Appx:Proof of winning}.
\end{proof}

Following Schmidt's idea, we will show that $\Hc$-absolute winning sets possess countable intersection property:

\begin{Prop} \lab{Prop:Countable Intersection}
	Let $Y$ be an $(\Hc, \beta)$-diffuse set. Then a countable intersection of sets $\Hc$-absolute winning on $Y$ is also $\Hc$-absolute winning on $Y$.
\end{Prop}

\begin{proof}
	See~\S\ref{Appx:Proof of Countable Intersection}.
\end{proof}

An easy consequence of the countable intersection property 
is that the $\Hc$-absolute winning property remains if we discard a countable number of removable points from the target set:

\begin{Prop} \lab{Prop:Remove Countable}
	Let $Y$ be an $\Hc$-diffuse set, $W$ be an $\Hc$-absolute winning on $Y$, and $Z \subseteq W \cap \overline{\bigcup_{\Lc \in \Hc} \prt{\Lc \cap Y}}$ is countable.  Then $W \smallsetminus Z$ is $\Hc$-absolute winning on $Y$. 
\end{Prop}

\begin{proof}
	See~\S\ref{Appx:Proof of Remove Countable}.
\end{proof}

Probably the most importance property of diffuse sets is the inheritance property whose proof is verbatim to the proof of~\cite[Proposition 4.9]{BroderickFishmanKleinbockReichWeiss12} for the $k$-dimensional absolute game:

\begin{Prop}[{\cite[Proposition 4.9]{BroderickFishmanKleinbockReichWeiss12}}] \lab{Prop:Inheritance}
	If $Y \subseteq Z$ are both $\Hc$-diffuse and $W$ is $\Hc$-absolute winning on $Z$, then $W$ is also $\Hc$-absolute winning on $Y$.
\end{Prop}

In the proof of Theorem~\ref{Thm:Main}, we will use a version of the $\Hc$-absolute game, in which Alice is allowed to choose $N$ closed subsets $\Lc_1, ..., \Lc_N$ in $\Hc$ in each move:
\eq{}{A_n = \prt{\bigcup_{i = 1}^N \Lc_i, \rho}}
for a fixed $N \geq 1$.  Let
\eq{}{\Hc^{*N} = \cbrk{\bigcup_{i = 1}^N \Lc_i : \Lc_i \in \Hc \text{ for } 1 \leq i \leq N}.}
We will show that this change in the rule won't affect the class of $\Hc$-absolute winning sets:

\begin{Prop} \lab{Prop:NHWinning}
	Assume that $Y$ is $\Hc$-diffuse, and let $W \subseteq X$.  Then the followings are equivalent:
	\begin{itemize}
		\item[\rm (i)]	$W$ is $\Hc$-absolute winning on $Y$.
		\item[\rm (ii)]	$W$ is $\Hc^{*N}$-absolute winning on $Y$.
	\end{itemize}
\end{Prop}

\begin{proof}
	See~\S\ref{Appx:Proof of NHWinning}.
\end{proof}

\begin{Rem} \lab{Rem:positional strategy}
	Technically speaking, a strategy for Alice in the game described above will have to take all the previous moves in consideration.  Nevertheless, it follows from~\cite[Theorem 7]{Schmidt66} that for a general topological infinite game of two players, including both games described in~\S\ref{Subsec:Schmidt game} and~\S\ref{Subsec:H-absolute game}, if Alice is winning then there exists a winning strategy for Alice that only takes Bob's immediately preceding move into account.  Such strategy is called a \emph{positional winning strategy}, and for our interest in the case of the $\Hc$-absolute game, it can be defined as a function $\sigma: Y \times \R_{>0} \to \Hc \times \R_{>0}$ satisfying: 
	\begin{itemize}
		\item[(i)]	For any $B_n \in Y \times \R_{>0}$, $A_n = \sigma(B_n)$ satisfies~\ref{H2},
		\item[(ii)]	For any sequence $B_1, A_1 = \sigma(B_1), B_2, A_2 = \sigma(B_2), ...$ satisfying \ref{H1}--\ref{H4}, Bob always has available choices at every stage of the game, and the intersection:
		\eq{}{W \cap Y \bigcap_{n = 1}^\infty \B(B_n) \neq \emptyset.}
	\end{itemize}
\end{Rem}

We end this section by proving Corollaries~\ref{Cor:ESK} and  \ref{Cor:EGL}  
assuming Theorem~\ref{Thm:Main}.

\begin{proof}[Proof of Corollary~\ref{Cor:EGL} and Theorem~\ref{Thm:EGL}]
	Recall that for every $t \in [0, 1]$ we have defined
	\eq{}{T(t) = \cbrk{v \in S: \phi'_v(t) \neq 0}.}
	Let
	\eq{}{D := \cbrk{t \in [0, 1]: \# T(t)_\R + 2 \prt{\# T(t)_\CO} \leq \frac{d}{2} }. }
	First, we will assume that $D = \varnothing$ and $\phi$ is injective.  Then by Proposition~\ref{Prop:H-diffuse curve}, $\phi([0, 1])$ is $\Hc_K$-diffuse.  So by Theorem~\ref{Thm:Main} and Proposition~\ref{Prop:Inheritance}, $\BA_K$ is $\Hc_K$-absolute winning on $\phi([0, 1])$.  

		Let $a, b, c$ be defined as in the proof of Proposition~\ref{Prop:H-diffuse curve}, and let $0 < \beta <\frac{1}{3}$ be sufficiently small such that $\phi([0, 1])$ is $\prt{\Hc_K, \frac{2\beta'}{1 - \beta'}}$-diffuse for $\beta' = \beta \cdot \min \cbrk{1, \frac{ac}{4b\sqrt{d}}}$.  In particular, $\BA_K$ is $\prt{\Hc_K, \beta'}$-absolute winning on $\phi([0, 1])$.  Let $\sigma: \phi([0, 1]) \times \R_{>0} \to \Hc_K \times \R_{>0}$ be a positional winning strategy (see Remark~\ref{Rem:positional strategy} above). 
	Consider the $\prt{\Hc, \beta}$-absolute game on $[0, 1]$ with $\Hc = \cbrk{\cbrk{t}: 0 \leq t \leq 1}$, and let $B_n = \prt{t_n, r_n} \in [0, 1]\times \R_{>0}$ be Bob's arbitrary move.  It suffices to assume that $r_n b\sqrt{d}  < \rho_{\phi(t_n)}$.

	For every $t_n \leq t \leq t_n + r_n$, the arc length of $\phi([t_n, t])$ is:
	\eq{}{l \prt{\phi([t_n, t])} \leq (t - t_n) b\sqrt{d} \leq r_n b\sqrt{d}.}
	That implies:
	\eq{}{\phi \prt{\B(B_n)} \subseteq \B \prt{\phi(t_n), r_n b\sqrt{d} }.}
	Let
	\eq{}{B'_n= \prt{\phi(t_n), r_n b\sqrt{d} } \text{ \quad and \quad} A'_n = \sigma(B'_n) = \prt{\Lc_n, \rho'_n}.}
	
	By the Mean Value Theorem, $\Lc_n$ intersects $\phi([0, 1])$ at at most 1 point.  If $$\phi^{-1} \prt{\Lc_n^{\prt{\rho'_n}} \cap \phi([0, 1])} \cap \B(B_n) = \varnothing$$ then Alice can make an arbitrary move $A_n$ in the $(\Hc, \beta)$-absolute game on $[0, 1]$.  If $\Lc_n \cap \B(B_n) \neq \varnothing$, then let $A_n = \prt{\phi^{-1} \prt{\Lc_n \cap \B(B_n)}, \beta r_n}$, otherwise, let $A_n = \prt{\phi^{-1} \prt{\x}, \beta r_n}$ for arbitrary $\x \in \Lc_n^{\prt{\rho'_n}} \cap \B(B'_n) \cap \phi([0, 1])$.  So for any $\y \in \Lc_n^{\prt{\rho'_n}} \cap \B(B'_n) \cap \phi([0, 1])$,
	\eq{}{\abs{c(A_n) - \phi^{-1}(\y)} \leq \frac{2}{ac} \norm{\pi_{\Lc^\perp} \prt{\y - \phi(c(A_n))}} \leq \frac{2}{ac} 2\beta' r_n b \sqrt{d} \leq \beta r_n.}
	It shows that:
	\eq{}{\phi\prt{\B(A_n)} \supseteq \Lc_n^{\prt{\rho'_n}} \cap \B(B'_n) \cap \phi([0, 1]).}
	Moreover, since $\beta' \leq \beta$,
	\eq{}{\norm{\phi(t_{n + 1}) - \phi(t_n)} \leq (1 - \beta) r_n b \sqrt{d} \leq (1 - \beta') r_n b \sqrt{d}.}
	So the sequence $B'_n, \sigma(B'_n), B'_{n + 1}, \sigma \prt{B'_{n + 1}}, ...$ satisfies the conditions~\ref{H1}--\ref{H5} of the $\prt{\Hc_K, \beta'}$-absolute game on $\phi([0, 1])$.  Therefore,
	\eq{}{\phi \prt{\bigcap_{n = 1}^\infty \B(B_n)} \subseteq \bigcap_{n = 1}^\infty \B(B'_n) \cap \phi([0, 1]) \subseteq \BA_K \cap~\phi([0, 1]).}
	Hence, $\phi^{-1}\prt{\BA_K}$ is absolute winning on $[0, 1]$.

	In the second case when $D = \varnothing$ and $\phi$ is not injective, by the Constant Rank Theorem, we can cover the interval $[0, 1]$ so that $\phi$ is injective on each subinterval.  Applying the first case,  we have that $\phi([0, 1])$ is $\Hc_K$-diffuse and $\phi^{-1}(\BA_K)$ is absolute winning.
	
	Finally, when $D \neq \varnothing$, then by the previous case, for every $n \geq 1$: 
	\eq{}{W_n := \phi^{-1}(\BA_K) \cup \bigcup_{t \in D} \prt{\prt{t - \frac{1}{n}, t + \frac{1}{n}} \cap [0, 1]}}
	is absolute winning on $[0, 1]$.  
	
	Since
	\eq{}{D	= \cbrk{t \in [0, 1]: \# T(t)_\R + 2 \prt{\# T(t)_\CO} \leq \frac{d}{2} } = \bigcup_{\substack{T \subseteq S \\ \#T_\R + 2 \prt{\#T_\CO} > \frac{d}{2}}} \bigcap_{v \in T} \prt{\phi'_v}^{-1} \prt{0},}
	$D$ is closed, and hence,
	\eq{}{\phi^{-1} \prt{\BA_K} \cup D = \bigcap_{n = 1}^\infty W_n.} 
	Applying Proposition~\ref{Prop:Countable Intersection}, $\phi^{-1}\prt{\BA_K} \cup D$ is absolute winning on $[0, 1]$.  Thus, by Proposition~\ref{Prop:Remove Countable}, $\phi^{-1}\prt{\BA_K}$ is absolute winning on $[0, 1]$. 

\end{proof}

\begin{Rem}
	The corollary still holds with $[0,1]$ replaced by $\R$.
\end{Rem}

\begin{proof}[Proof of Corollary~\ref{Cor:ESK} and Theorem~\ref{Thm:ESK}]
	Let $K$ be an imaginary quadratic field and $Y$ be the support of an Ahlfors regular measure.  Consider McMullen's absolute game on $\CO$ where $\Hc = \cbrk{\cbrk{x}: x \in \CO}$.  Since $\BA_K$ is absolute winning by Corollary~\ref{Cor:AW} and $Y$ is $\Hc$-diffuse by Example~\ref{Eg:Ahlfors}, it follows from Proposition~\ref{Prop:Inheritance} that $\BA_K$ is absolute winning on $Y$.  Thus, $\BA_K \cap Y$ is winning for the Schmidt game playing on $Y$ by Proposition~\ref{Prop:winning}.
\end{proof}



\section{Proof of Theorem~\ref{Thm:Main}} \lab{Sec:Proof}
One of the key ingredients to show the winning property 
has traditionally been the Simplex Lemma, see \cite{BroderickFishmanKleinbockReichWeiss12} and references therein.  The next statement (Lemma~\ref{Lem:Unique1}) can be thought of as a number field analogue of the Simplex Lemma. Together with an estimate of the growth of the height function with respect to the flow $g_t \Lambda_\x$ as $\x$ varies in Lemma~\ref{Lem:Bounding H}, it implies Lemma~\ref{Lem:Unique2} which is a more refined version of the Simplex Lemma. 

\begin{Lem}[{\cite[Lemma 4.1]{EinsiedlerGhoshLytle13}}] \lab{Lem:Unique1}
	Let $u \in G$, and $\vect{z} = u \iv[S]{\vvect{p}{q}}, \vect{z}' = u \iv[S]{\vvect{p'}{q'}}$, where $p, p' \in \BO$, $q, q' \in \BO \smallsetminus \cbrk{0}$.  If $\height{\vect{z}} \height{\vect{z}'} < 2^{-d}$, then $\frac{p}{q} = \frac{p'}{q'}$. 
\end{Lem}

For completeness, we will provide a proof of Lemma~\ref{Lem:Unique1} here, which is the same as the proof in~\cite{EinsiedlerGhoshLytle13} with a minor correction on the constant.

\begin{proof}
	Let $P = \iv[S]{\begin{matrix} p &p' \\ q &q' \end{matrix}} \in M_{2, 2} (K_S)$, then the height of the determinant of $uP$ is:
	\begin{align*}
		\height{\det \prt{uP}}	&= \height{\det \prt{P}}	&(\text{since } u \in \SL_2(K_S)) \\
			&= \height{\iv[S]{p q' - p' q}} \\
			&= \abs{N(p q' - p' q)}	&(\text{since } p q' - p' q \in \BO)
	\end{align*}
	On the other hand,
	\begin{align*}
		\height{\det \prt{uP}}	&=  \height{\det \begin{pmatrix} \z_1 &\z'_1 \\ \z_2 &\z'_2 \end{pmatrix}} \\
			&= \height{\z_1 \z_2' - \z'_1 \z_2} \\
			&= \prod_{v \in S} \abs{(\z_1)_v (\z_2')_v - (\z'_1)_v (\z_2)_v}^{d_v} \\
			&\leq \prod_{v \in S} \prt{2 \max \cbrk{\abs{(\z_1)_v}, \abs{(\z_2)_v}} \cdot \max \cbrk{\abs{(\z'_1)_v}, \abs{(\z'_2)_v}}}^{d_v} \\
			&= 2^d \height{\vect{z}} \height{\vect{z}'} < 1.
	\end{align*}
	Thus, $|N(pq' - p'q)| = 0$ and $pq' = p'q$, since  {$\abs{N(pq' - p'q)} \in \N$}. 
\end{proof}

\begin{Lem} \lab{Lem:Bounding H}
	Let $\x, \y \in K_{S}$ such that $\norm{\x - \y} \leq \rho$.  Then for any $p \in \BO$, $q \in \BO \smallsetminus \cbrk{0}$ and for any $t \geq 0$,
	\eq{}{
		\prt{1 + e^{2t} \rho}^{-d} \height{g_t u_\x \iv[S]{\vvect{p}{q}}} \leq \height{g_t u_\y \iv[S]{\vvect{p}{q}}} \leq \prt{1 + e^{2t} \rho}^d \height{g_t u_\x \iv[S]{\vvect{p}{q}}}.
	}
\end{Lem}
	
\begin{proof}
	\begin{align*}
		\frac{\height{g_t u_\y \iv[S]{\vvect{p}{q}}}}{\height{g_t u_\x \iv[S]{\vvect{p}{q}}}} &= \prod_{v \in S} \frac{ \max \cbrk{e^{-t} \abs{\iv[v]{q}}, e^{t} \abs{\iv[v]{q} y_v + \iv[v]{p}}}^{d_v} }{ \max \cbrk{e^{-t} \abs{\iv[v]{q}}, e^{t} \abs{\iv[v]{q} x_v + \iv[v]{p}}}^{d_v} } \\
			&= \prod_{v \in S} \frac{ \max \cbrk{ e^{-t}, e^{t} \abs{ y_v + \iv[v]{\dfrac{p}{q}} } }^{d_v} }{ \max \cbrk{ e^{-t}, e^{t} \abs{ x_v + \iv[v]{\dfrac{p}{q}} } }^{d_v} } \\
			&\leq \prod_{v \in S} \frac{ \max \cbrk{ e^{-t}, e^{t} \prt{ \abs{y_v - x_v} + \abs{ x_v + \iv[v]{\dfrac{p}{q}} }}}^{d_v} }{ \max \cbrk{ e^{-t}, e^{t} \abs{ x_v + \iv[v]{\dfrac{p}{q}} }}^{d_v} } \\
			&\leq \prod_{v \in S} \frac{ \max \cbrk{ e^{-t}, e^{t} \prt{ \rho + \abs{ x_v + \iv[v]{\dfrac{p}{q}} }}}^{d_v} }{ \max \cbrk{ e^{-t}, e^{t} \abs{ x_v + \iv[v]{\dfrac{p}{q}} }}^{d_v} } \\
			&\leq \prod_{v \in S} \prt{1 + e^{2t} \rho}^{d_v} \\
			&= \prt{1 + e^{2t}\rho}^d.
	\end{align*}
	
	The reverse inequality is obtained by symmetry.
\end{proof}

\begin{Lem}	\lab{Lem:Unique2}
	Let $\B = \B(\y, \rho)$.  Let $t > 0$ and $\varepsilon = 2^{-d} (1 + 2e^{2t}\rho)^{-d}$. If there exists $p \in \BO$, $q \in  \BO\smallsetminus \cbrk{0}$ and $\x \in \B$ such that
	\eq{}{\height{g_t u_{\x} \iv[S]{\vvect{p}{q}}} \leq \varepsilon,}
	then for every $p' \in \BO$, $q' \in \BO \smallsetminus \cbrk{0}$ with $\frac{p'}{q'} \neq \frac{p}{q}$, and for every $\x' \in \B$:
	\eq{}{\height{g_t u_{\x'} \iv[S]{\vvect{p'}{q'}}} \geq 1.}
\end{Lem}

\begin{proof}
	Assume by contradiction that there exists $\x' \in \B$, $p' \in \BO$, $q' \in \BO \smallsetminus \cbrk{0}$, $\frac{p}{q} \neq \frac{p'}{q'}$, such that
	\eq{}{\height{g_t u_{\x'} \iv[S]{\vvect{p'}{q'}}} < 1.}
	
	Then by Lemma~\ref{Lem:Bounding H}, 
	\eq{}{\height{g_t u_{\x} \iv[S]{\vvect{p'}{q'}}} \leq \height{g_t u_{\x'} \iv[S]{\vvect{p'}{q'}}} \cdot \prt{1 + e^{2t} \norm{\x - \x'} }^d  < \prt{1 + 2 e^{2t} \rho}^d.}
	
	That implies
	\eq{}{\height{g_t u_{\x} \iv[S]{\vvect{p}{q}}} \cdot \height{g_t u_{\x} \iv[S]{\vvect{p'}{q'}}} < \varepsilon \cdot \prt{1 + 2 e^{2t}\rho}^d = 2^{-d},}
	contradicting Lemma~\ref{Lem:Unique1}.
\end{proof}

Let $f(t)$ be a function of $t$. Denote the forward derivative of $f$ at $t$ to be:
\eq{Forward Der}{\frac{d}{dt}^+ (f(t)) = \lim_{h \to 0^+} \frac{f(t + h) - f(t)}{h},}
then for any $\x \in K_{S}$, $p \in \BO$, $q \in \BO \smallsetminus \cbrk{0}$, $\frac{d}{dt}^+ \height{g_t u_\x \iv[S]{\vvect{p}{q}}}$ exists for every $t \geq 0$.  Note that the two-sided derivative $\frac{d}{dt} \height{g_t u_\x \iv[S]{\vvect{p}{q}}}$ will fail to exist at the time $t$ for which $e^{-t} = e^t \abs{\iv[v]{q}x_v + \iv[v]{p}}$ for some $v \in S$.

Moreover, if we denote 
\eq{Tx}{T_{\x, t} = \cbrk{v \in S: \abs{x_v + \iv[v]{\frac{p}{q}}} < e^{-2t}},}
then
\begin{align}
\begin{split} \lab{eq:H Der}
	\frac{d}{dt}^+ \log \height{g_t u_\x \iv[S]{\vvect{p}{q}}} &= \frac{d}{dt}^+ \log \prod_{v \in S} \max \cbrk{e^{-t} \abs{\iv[v]{q}}, e^t \abs{\iv[v]{q} x_v + \iv[v]{p}}}^{d_v} \\
		&= \sum_{v \in S} d_v \frac{d}{dt}^+ \log \max \cbrk{e^{-t} \abs{\iv[v]{q}}, e^t \abs{\iv[v]{q} x_v + \iv[v]{p}}} \\
		&= \sum_{v \in S} d_v \frac{d}{dt}^+ \prt{\log \max \cbrk{e^{-t}, e^t \abs{x_v + \iv[v]{\frac{p}{q}}}} + \log \abs{\iv[v]{q}}} \\
		&= \sum_{v \in S \smallsetminus T_{\x, t}} d_v \frac{d}{dt}^+ {\log e^{t} \abs{x_v + \iv[v]{\frac{p}{q}}}} + \sum_{v \in T_{\x, t}} d_v \frac{d}{dt}^+ {\log e^{-t}} \\
		&= \sum_{v \in S \smallsetminus T_{\x, t}} d_v - \sum_{v \in T_{\x, t}} d_v.
\end{split}
\end{align}

This gives us a trivial bound of the forward derivatives:
\eq{H Der Bounded}{\abs{\frac{d}{dt}^+ \log \height{g_t u_\x \iv[S]{\vvect{p}{q}}}} \leq d.}

\begin{Rem} \lab{Rem:HDer nonneg}
	It also follows from~\eqr{H Der} that  
	\eq{}{\# \prt{T_{\x, t}}_\R + 2 \cdot \# \prt{T_{\x, t}}_\CO \leq \frac{d}{2}} 
	if and only if
	\eq{}{\frac{d}{dt}^+ \log \height{g_t u_\x \iv[S]{\vvect{p}{q}}} \geq 0.}
\end{Rem}

From this remark, we can deduce that every subspace of $\Hc_K$ defined in~\eqr{Hc} does not contain any badly approximable $S$-numbers:

\begin{Prop} \lab{Prop:Optimal}
	For any $\Lc \in \Hc_K$,
	\eq{}{\Lc \cap \BA_K = \varnothing.}
\end{Prop}

\begin{proof}
	Let $\y \in \Lc = \Lc(\x, T) \in \Hc_K$ be arbitrary, and let $p, q \in \BO$ such that $\x = \iv[S]{\frac{p}{q}}$.  Since for every $v \in T$,
	\eq{}{\abs{\y_v - \iv[v]{\frac{p}{q}}} = \abs{\y_v - \x_v} = 0,}
	we have that for every $t \geq 0$,
	\eq{}{T \subseteq T_{\y, t}.}
	In particular,
	\eq{}{\# \prt{T_{\y, t}}_\R + 2 \cdot \# \prt{T_{\y, t}}_\CO > \frac{d}{2}.} 
	So by Remark~\ref{Rem:HDer nonneg}, for every $t \geq 0$,
	\eq{}{\frac{d}{dt}^+ \log \height{g_t u_\y \iv[S]{\vvect{p}{q}}} \leq -1.}
	That implies
	\eq{}{\lim_{t \to \infty} \height{g_t u_\y \iv[S]{\vvect{p}{q}}} = 0.}
	Thus, $\y \notin \BA_K$.
\end{proof}

Moreover, the forward derivatives provide us another criterion for an $\x \in K_{S}$ to be badly approximable, which will be used in the proof of Theorem~\ref{Thm:Main}:
\begin{Lem}	\lab{Lem:Bad}
	 $\x \in \BA_{K}$ if and only if there exists $\varepsilon > 0, c > 0$, and a sequence $0 < t_1 < t_2 < ...$ with $\lim\limits_{n \to \infty} t_n = \infty$ and $t_{n + 1} - t_n \leq c$ such that for every $p, q \in \BO, (p, q) \neq (0, 0)$,
	\eq{BA flow}{\text{either } \height{g_{t_n} u_\x \iv[S]{\vvect{p}{q}}} \geq \varepsilon \text{\quad or \quad} \frac{d}{dt}^+ \log \height{ g_{t_n} u_\x \iv[S]{\vvect{p}{q}}} \geq 0.}
\end{Lem}

\begin{proof}
	If $\x \in \BA_{K}$, let $\varepsilon = \inf\cbrk{\delta_H \prt{g_t \Lambda_x}: t \geq 0} $; it is positive by Proposition~\ref{Prop:Dani Corr}.  Then clearly $\varepsilon$ satisfies~\eqr{BA flow} for every sequence $t_n$.
	For the converse, since $\height{g_0 u_\x \iv[S]{\vvect{p}{q}}} \geq 1$ and by~\eqr{H Der Bounded}, for every $p, q \in \BO$ with $(p, q) \neq (0, 0)$ and for every $t \geq 0$:
	\eq{}{
		\height{g_t u_\x \iv[S]{\vvect{p}{q}}} \geq \min \cbrk{e^{-dt_1}, \inf_{n \geq 1} \prt{\varepsilon e^{-d \prt{t_{n + 1} - t_n}}}} \geq \min \cbrk{e^{-dt_1}, \varepsilon e^{-dc}} > 0.
	}
	Hence, $\x \in \BA_{K}$.
\end{proof}

We end this section with the proof of Theorem~\ref{Thm:Main}.  As usual, we will provide a strategy for Alice, and show that  it is indeed a winning strategy for $\BA_{K}$.

\begin{proof}[Proof of Theorem~\ref{Thm:Main}]
	We remark that since $K_S$ is a real Banach space, $(\x, r)$ and $(\Lc, \rho)$ are uniquely defined by $\B(\x, r)$ and $\Lc^{(\rho)}$ respectively. 
	So to ease the notation, we will identify $B_n$ and $A_n$ with the corresponding sets, and use $\B_n$ and $\mathbf{A}_n$ for this identification in this proof. 
	
	Fix $\beta > 0$, and let $\B_n = \B(\x_n, \rho_n)$ be Bob's arbitrary $n^{\text{th}}$ move.  Without loss of generality, assume that $\rho_1 < 1$ and $\rho_n \to 0$.  Let
	\eq{tn}{t_n = -\frac{1}{2} \log \prt{\beta \rho_n},}
	and let
	\eq{eps}{\varepsilon = 2^{-d} \prt{1 + \frac{2}{\beta}}^{-d}.}
	Then
	\eq{tn bounded}{t_{n + 1} \leq t_n + \frac{1}{2} \log \prt{\frac{1}{\beta}} \text{ since } \rho_{n + 1} \geq \beta \rho_n.}
	
	If at $n^{\text{th}}$ stage of the game, $\delta_H \prt{g_{t_n} \Lambda_\x} \geq \varepsilon$ for every $\x \in \B_n$, then Alice can make arbitrary move.  Otherwise, let $\x \in \B_n$, $p_n \in \BO$, $q_n \in \BO \smallsetminus \cbrk{0}$ such that:
	\eq{}{\height{g_{t_n} u_\x \iv[S]{\vvect{p_n}{q_n}}} < \varepsilon,}
	then by Lemma~\ref{Lem:Unique2} and Remark~\ref{Rem:HDer nonneg}, Alice only has to worry about those subspaces in $\Hc_K$ passing through $\iv[S]{\frac{p_n}{q_n}}$.  
	
	For her $n^{\text{th}}$ move, Alice will pick the $(\beta \rho_n)$-neighborhood of all the subspaces in $\Hc_K$ passing through $\iv[S]{\frac{p_n}{q_n}}$:
	\eq{An}{\mathbf{A}_n = \bigcup_{\# T_\R + 2(\# T_\CO) > \frac{d}{2}} \hplane{\iv[S]{\frac{p_n}{q_n}}}{T}^{\prt{\beta \rho_n}}.}
	Since $\beta \rho_n = e^{-2t_n}$, it follows from~\eqr{H Der} and Remark~\ref{Rem:HDer nonneg} that, for every $\x \in \B_n \smallsetminus \mathbf{A}_n$ and for every $p, q \in \BO$, $(p, q) \neq (0, 0)$:
	\eq{}{
		\text{either } \height{g_{t_n} u_\x \iv[S]{\vvect{p}{q}}} \geq \varepsilon \text{ or }  \frac{d}{dt}^+ \log \height{ g_{t_n} u_\x \iv[S]{\vvect{p}{q}}} \geq 0.
	}
	So $\x_\infty = \bigcap_{n = 1}^\infty B_n$ satisfies the conditions of Lemma~\ref{Lem:Bad}, and hence, $\x_\infty \in \BA_{K}$.  Thus, by Proposition~\ref{Prop:NHWinning}, $\BA_{K}$ is $\Hc_K$-absolute winning.
\end{proof}







\section*{Appendix A.\texorpdfstring{\quad}{ } 
Hattori's approach to badly approximable \texorpdfstring{$S$}{S}-numbers} \lab{Sec:Appx A}
\renewcommand{\thesection}{A}
\stepcounter{section}
\setcounter{Thm}{0}
\setcounter{equation}{0}
\addtocontents{toc}{\protect\setcounter{tocdepth}{1}}
Hattori~\cite{Hattori07} proved the following version of Dirichlet's Theorem:

\begin{Thm}[{\cite[Theorem 1, 2]{Hattori07}}] \lab{Thm:Hattori DT}
	If $K$ be a real quadratic or totally complex quartic number field, then there is a constant $C = C_K > 0$ depending only on $K$ such that for every $\x \in K_S \smallsetminus \iv[S]{K}$, there are infinitely many $p \in \BO$, $q \in \BO \smallsetminus \cbrk{0}$ satisfying:
	\eq{}{\norm{\x + \iv[S]{\frac{p}{q}}} \leq \norm{\iv[S]{q}}^{-2}.}.
\end{Thm}

Hence, we say that $\x \in K_{S} \smallsetminus \iv[S]{K}$ \emph{badly approximable in Hattori's sense} if there exists $c > 0$ such that for every $p \in \BO$, $q \in \BO \smallsetminus \cbrk{0}$,
\eq{}{\norm{\x + \iv[S]{\frac{p}{q}}} > c \norm{\iv[S]{q}}^{-2},}
or equivalently,
\eq{Hat BA-inf}{\inf \cbrk{ \norm{\iv[S]{q}}^2 \cdot \norm{\x + \iv[S]{\frac{p}{q}}} : p, q \in \BO, q \neq 0 } > 0.}
The set of $S$-numbers badly approximable in Hattori's sense is denoted by $\BA'_K$.

\begin{Prop} \lab{Prop:Same BA}
	Let $K$ be a real quadratic $(d = 2 = \#S)$ or totally complex quartic $(d = 4 = 2\#S)$ number field; then $\BA_K = \BA'_K$.
\end{Prop}

\begin{proof}
	Recall that $\x \in \BA_{K}$ if and only if \eqr{BA in K-inf} holds. 
	Then clearly $\BA_K \subseteq \BA'_K$, since:
	\begin{align*}
		\norm{\iv[S]{q} \cdot \x + \iv[S]{p}} &= \max_{v \in S} \abs{\iv[v]{q} x_v + \iv[v]{p}} 
			= \max_{v \in S} \abs{\iv[v]{q} \prt{x_v + \iv[v]{\frac{p}{q}}}} \\
			&\leq \max_{v \in S} \abs{\iv[v]{q}} \cdot \max_{v \in S} \abs{x_v + \iv[v]{\frac{p}{q}}} 
			= \norm{\iv[S]{q}} \cdot \norm{ \x + \iv[S]{\frac{p}{q}}}.
	\end{align*}
	
	For the converse, first by using Lemma~\ref{Lem:Height properties}, for every $p \in \BO, q \in \BO \smallsetminus \cbrk{0}$, there exists a unit $\xi \in \BO^\times$ such that:
	\eq{}{
		\abs{N(q)} \cdot \max_{v \in S} \abs{x_v + \iv[v]{\frac{p}{q}}}^{d_v} \geq C^{-d} \norm{\iv[S]{\xi q}}^d \cdot \max_{v \in S} \abs{x_v + \iv[v]{\frac{p}{q}}}^{d_v} = C^{-d} \prt{\norm{\iv[S]{\xi q}}^2 \cdot \norm{x_v + \iv[v]{\frac{\xi p}{\xi q}}}}^\frac{d}{2} .
	}
	In particular, if $\x \in \BA_K'$ then
	\eq{BA inf 2}{\inf \cbrk{ \abs{N(q)} \cdot \max_{v \in S} \abs{x_v + \iv[v]{\frac{p}{q}}}^{d_v} : p, q \in \BO, q \neq 0 } > 0.}	
	Hence, it suffices to show that if $\x \in K_S \smallsetminus \iv[S]{K}$ satisfying~\eqr{BA inf 2}, then $\x \in \BA_{K}$.
	
	Unwrapping Proposition~\ref{Prop:Dani Corr}, we see that $\x \in \BA_{K}$ if and only if
	\eq{BA unwrapped}{\inf \cbrk{ \prod_{v \in S} \max \cbrk{ \abs{e^{-t} \iv[v]{q}}, \abs{e^{t} \prt{\iv[v]{q} x_v + \iv[v]{p}}}}^{d_v} : (p, q) \in \BO^2 \smallsetminus (0, 0), t \geq 0 } > 0.}

	When $q = 0, p \neq 0$,
	\eq{}{
		\prod_{v \in S} \max \cbrk{ \abs{e^{-t} \iv[v]{q}}, \abs{e^{t} \prt{\iv[v]{q} x_v + \iv[v]{p}}}}^{d_v} = \prod_{v \in S} \abs{e^t \iv[v]{p}}^{d_v} \geq |N(p)| \geq 1.
	}
	So it suffices to consider the case when $q \neq 0$ and the product in the left hand side of the above formula is $< 1$. In that case,
	\eq{prod 1}{\prod_{v \in S} \max \cbrk{ \abs{e^{-t} \iv[v]{q}}, \abs{e^{t} \prt{\iv[v]{q} x_v + \iv[v]{p}}}}^{d_v} = \abs{N(q)} \prod_{v \in S} \max \cbrk{ e^{-t}, e^t \abs{ x_v + \iv[v]{\frac{p}{q}}}}^{d_v}.}

	Since $\#S = 2$ and the $d_v$'s are both 1's or both 2's,
	\eq{inf 1}{\inf \cbrk{\prod_{v \in S} \max \cbrk{ e^{-t}, e^t \abs{ x_v + \iv[v]{\frac{p}{q}}}}^{d_v} : t \geq 0} = \max_{v \in S} \abs{x_v + \iv[v]{\frac{p}{q}}}^{d_v}.}

	Combining~\eqr{inf 1} and~\eqr{prod 1}, we have the equivalence of~\eqr{BA unwrapped} and~\eqr{BA inf 2}.
\end{proof}

\begin{Rem}
	Combining the proof of Proposition~\ref{Prop:Same BA} with Theorem~\ref{Thm:Weak Dirichlet} gives us another proof of Theorem~\ref{Thm:Hattori DT} of Hattori.
\end{Rem}

\section*{Appendix B.\texorpdfstring{\quad}{ } Proofs of properties of \texorpdfstring{$\Hc$}{H}-absolute winning sets} \lab{Sec:Appx B}
\renewcommand{\thesection}{B}
\stepcounter{section}
\setcounter{Thm}{0}
\setcounter{equation}{0}
\addtocontents{toc}{\protect\setcounter{tocdepth}{1}}
Since our definitions are slightly different from those found in~\cite{Schmidt66, BroderickFishmanKleinbockReichWeiss12}, we provide the proofs of basic properties of $\Hc$-absolute winning sets in this appendix for completeness. 

\subsection{Proof of Lemma~\ref{Lem:Extend beta}}
\lab{Appx:Proof of Extend beta}

	Let $0 < \beta' \leq \frac{\beta}{2 + \beta}$, and let $\sigma_A$ be an $(\Hc, \beta")$-absolute positional winning strategy for Alice for some $0 < \beta" \leq \beta'$.  Since $Y$ is $\beta$-diffuse, Lemma~\ref{Lem:diffuse} will guarantee that if Alice is using the $\sigma_A$ strategy in an $(\Hc, \beta')$-absolute game on $Y$, Bob will always have eligible moves, and the game will last infinitely.  And hence, it is a winning strategy for Alice in the $(\Hc, \beta')$-absolute game.
\qed

\subsection{Proof of Proposition~\ref{Prop:winning}}
\lab{Appx:Proof of winning}

	Let $0 < \alpha \leq \frac{\beta}{2 + \beta}$ and $0 < \beta' < 1$, and denote $\gamma = \alpha \beta' < \frac{\beta}{2 + \beta}$.  By Lemma~\ref{Lem:Extend beta}, $W$ is $(\Hc, \gamma)$-absolute winning.  Let $\sigma_A: Y\!\times\!\R_{>0} \to \Hc\!\times\!\R_{>0}$ be a positional $(\Hc, \gamma)$-absolute winning strategy for Alice.  With Remark~\ref{Rem:radius 0} and Lemma~\ref{Lem:diffuse}, for any $B_n$ with sufficiently small radius, if we denote $\sigma_A(B_n) = (\Lc_n, \rho_n)$, Alice can pick a ball $A_n \preceq B_n$ such that: 
	\eq{}{r(A_n) = \alpha r(B_n) \text{~\quad and \quad} \dist{c(A_n), \Lc_n} > 2\alpha r(B_n).}
	Then for any $B_{n + 1} \preceq A_n$ with 
	\eq{}{r(B_{n + 1}) = \beta' r(A_n) = \beta' \alpha r(B_n) = \gamma r(B_n),}
	$B_{n + 1}$ satisfies~\ref{H4}:
	\eq{}{\dist{c(B_{n + 1}), \Lc_n} \geq \dist{c(A_n), \Lc_n} - \alpha (1 - \beta') r(B_n) > 2 \gamma r(B_n) \geq r(B_{n + 1}) + \rho_n.}
	Therefore, the infinite sequence 
	$B_1, \sigma_A(B_1), B_2, \sigma_A(B_2), ...$ satisfies~\ref{H1}--\ref{H5} for the $(\Hc, \gamma)$-absolute game.  That implies:
	\eq{}{W \cap Y \cap \bigcap_{n = 1}^\infty \B(B_n) \neq \varnothing.}
	Thus, $W \cap Y$ is $\frac{\beta}{2 + \beta}$-winning.
\qed

\subsection{Proof of Proposition~\ref{Prop:Countable Intersection}}
\lab{Appx:Proof of Countable Intersection}

	We follow the proof of the countable intersection property of Schmidt games~\cite[Theorem 2]{Schmidt66}.  By part (i) of Lemma~\ref{Lem:Extend beta}, it suffices to assume that $W_1, W_2, ...$ be a countable collection of sets that are $(\Hc, \gamma)$-absolute winning on $Y$ with $\gamma = \frac{\beta}{2 + \beta}$.  For $i = 1, 2, ...$, let $\sigma_i: Y\!\times\!\R_{>0} \to \Hc\!\times\!\R_{>0}$ be a positional $(\Hc, \gamma^{2^i})$-absolute winning strategy for Alice with the target set $W_i$.
	
	Let $B_1 = (x, \rho) \in Y\!\times\R_{>0}$ be Bob's arbitrary first move.  By Remark~\ref{Rem:radius 0}, and Lemma~\ref{Lem:diffuse}, Bob always has legitimate moves regardless of Alice's choices following the rules of the $(\Hc, \gamma)$-absolute game.
	
	We define Alice's new strategy $\sigma$ to be:
	\eq{}{\sigma \prt{B_1, ..., B_{2^{i - 1} + (n - 1)2^i}} =  \sigma_i \prt{B_{2^{i - 1} + (n - 1)2^i}} \text{\quad for } n = 1, 2, 3, ...}
	It is easy to check that for $i = 1, 2, ...$, the sequence
	\begin{align*}
		B'_1 &= B_{2^{i - 1}}, \\
		A'_1 &= \sigma\prt{B_1, B_2, ..., B_{2^{i - 1}}} = \sigma_i\prt{B_{2^{i - 1}}} \\
		B'_2 &= B_{2^{i - 1} + 2^i} \\
		A'_2 &= \sigma\prt{B_1, B_2, ..., B_{2^{i - 1} + 2^i}} = \sigma_i\prt{B_{2^{i - 1} + 2^i}}\\
		... \\
		B'_n &= B_{2^{i - 1} + (n - 1)2^i} \\
		A'_n &= \sigma\prt{B_1, B_2, ..., B_{2^{i - 1} + (n - 1)2^i}} = \sigma_i\prt{B_{2^{i - 1} + (n - 1)2^i}} \\
		...
	\end{align*}
	satisfies \ref{H1}-\ref{H5} for $(\Hc, \gamma^{2^i})$-absolute game, and hence
	\eq{}{W_i \cap Y \cap \prt{\bigcap_{n = 1} \B(B_n)} = W_i \cap Y \cap \bigcap_{n = 1} \B(B_{2^{i - 1} + (n - 1)2^i}) \neq \varnothing.}
	
	Since $r(B_n) \to 0$, the unique intersection point of $\B(B_n)$'s must belongs to all $W_i$'s.  Thus,
	\eq{}{\prt{\bigcap_{i = 1}^\infty W_i} \cap Y \cap \prt{\bigcap_{i = 1}^\infty \B(B_i)}  \neq \varnothing,}
	and $\bigcap_{i = 1}^\infty W_i$ is $\Hc$-absolute winning.  Note that the strategy $\sigma$ is not a positional winning strategy.
\qed

\subsection{Proof of Proposition~\ref{Prop:Remove Countable}}
\lab{Appx:Proof of Remove Countable}
	Let $Z = \cbrk{z_1, z_2, ...}$, and denote $W_n = W \smallsetminus \cbrk{z_1, ..., z_n}$.  Since $Z \subseteq W \cap \overline{\bigcup_{\Lc \in \Hc} \prt{\Lc \cap Y}}$, for every $i = 1, 2, ..$ and for every $\rho_i > 0$, there exists $\Lc_i \in \Hc$ such that $z_i \in \Lc_i^{\prt{\rho_i}}$.  If we let $A_i = \prt{\Lc_i, \rho_i}$ for $1 \leq i \leq n$, and for $i > n$, $A_i$ follow a positional winning strategy for Alice in the $\Hc$-absolute game on $Y$ with the target set $W$, then
	\eq{}{W_n \cap Y \cap \bigcap_{n = 1}^\infty \B(B_n) \neq \varnothing.}
	
	Note that by diffuseness, Bob always has eligible moves.  So, $W_n$ is $\Hc$-absolute winning on $Y$ for all $n = 1, 2, ..$.  Therefore, by Proposition~\ref{Prop:Countable Intersection}, 
	\eq{}{W \smallsetminus Z = \bigcap_{n = 1}^\infty W_n}
	is $\Hc$-absolute winning on $Y$.
\qed

\subsection{Proof of Proposition~\ref{Prop:NHWinning}}
\lab{Appx:Proof of NHWinning}

	(i) $\Rightarrow$ (ii): Since $\Hc \subseteq \Hc^{*N}$, every choice in $(\Hc, \beta)$-absolute game is available in the $(\Hc^{*N}, \beta)$-absolute game.  So Alice can apply her $(\Hc, \beta)$-absolute winning strategy in the $(\Hc^{*N}, \beta)$-game.  
	
	For the converse, assume that $Y$ is $(\Hc, \beta)$-diffuse and consider the $(\Hc, \gamma)$-absolute game, where $\gamma = \frac{\beta}{2 + \beta}$.  Let $\sigma: Y\!\times\!\R_{>0} \to \Hc^{*N}\!\times\!\R_{>0}$ be a positional winning strategy for Alice in the $(\Hc^{*N}, \gamma^N)$-absolute game on $Y$.  For $k = 0, 1, ...$, consider Bob's $(kN + 1)^{\text{th}}$ move $B_{kN + 1}$ and let
	\eq{}{\sigma\prt{B_{kN + 1}} = \prt{\Lc_{k, 1} \cup ... \cup \Lc_{k, N}, \rho_k}.}
	Then for $1 \leq i \leq N$, at $(kN + i)^{\text{th}}$ move, Alice chooses
	\eq{}{A_{kN + i} = \prt{\Lc_{k, i}, \rho_k}.}
	It is easy to check that Alice's choices satisfies~\ref{H2}, and by Lemma~\ref{Lem:diffuse}, Bob can always make a move.  To see that this is a winning strategy for $(\Hc, \gamma)$-game, we view the following sequence:
	\eq{}{B'_1 = B_1, A'_1 = \sigma(B_1), B'_2 = B_{N + 1}, A'_2 = \sigma(B_{N + 1}), B'_3 = B_{2N + 1}, A'_3 = \sigma(B_{2N + 1}), ...}
	as a play in the $(\Hc^{*N}, \gamma^N)$-absolute game on $Y$.  It can be verified that this sequence satisfies~\ref{H1}--\ref{H5}, and since $\sigma$ is a $(\Hc^{*N}, \gamma^N)$-absolute winning strategy,
	\eq{}{W \cap Y \cap \bigcap_{n = 1}^\infty \B(B_n) = W \cap Y \cap \bigcap_{n = 1}^\infty \B(B'_n) \neq \varnothing.}
	Thus, $W$ is $\Hc$-absolute winning on $Y$.
\qed

\bibliographystyle{amsalpha}
\bibliography{mybib}

\newcommand{\etalchar}[1]{$^{#1}$}
\providecommand{\bysame}{\leavevmode\hbox to3em{\hrulefill}\thinspace}
\providecommand{\MR}{\relax\ifhmode\unskip\space\fi MR }
\providecommand{\MRhref}[2]{%
  \href{http://www.ams.org/mathscinet-getitem?mr=#1}{#2}
}
\providecommand{\href}[2]{#2}
\begin{thebibliography}{BFK{\etalchar{+}}12}

\bibitem[AGK15]{AnGuanKleinbock15}
J.~An, L.~Guan, and D.~Kleinbock, \emph{Bounded orbits of diagonalizable flows
  on $\operatorname{SL}_3(\mathbb{R})/\operatorname{SL}_3(\mathbb{Z})$}, to
  appear in Internat. Math. Res. Notices (2015), doi:10.1093/imrn/rnv120.

\bibitem[BFK{\etalchar{+}}12]{BroderickFishmanKleinbockReichWeiss12}
R.~Broderick, L.~Fishman, D.~Kleinbock, A.~Reich, and B.~Weiss, \emph{The set
  of badly approximable vectors is strongly ${C}^1$--incompressible}, Math.
  Proc. Cambridge Philos. Soc. \textbf{153} (2012), no.~2, 319--339.

\bibitem[BFS13]{BroderickFishmanSimmons13}
R.~Broderick, L.~Fishman, and D.~Simmons, \emph{Badly approximable systems of
  affine forms and incompressibility on fractals}, J. Number Theory
  \textbf{133} (2013), no.~7, 2186--2205.

\bibitem[Bur92]{Burger92}
E.~Burger, \emph{Homogeneous {D}iophantine approximation in {$S$}-integers},
  Pacific J. Math. \textbf{152} (1992), no.~2, 211--253.

\bibitem[Dan85]{Dani85}
S.~G. Dani, \emph{Divergent trajectories of flows on homogeneous spaces and
  diophantine approximation}, J. Reine Angew. Math. \textbf{359} (1985),
  55--89.

\bibitem[EGL13]{EinsiedlerGhoshLytle13}
M.~Einsiedler, A.~Ghosh, and B.~Lytle, \emph{Badly approximable vectors,
  ${C}^1$ curves and number fields}, arXiv:1401.0992, 2013.

\bibitem[ESK10]{EsdahlKristensen10}
R.~Esdahl-Schou and S.~Kristensen, \emph{On badly approximable complex
  numbers}, Glasg. Math. J. \textbf{52} (2010), 349--355.

\bibitem[FSU13]{FishmanSimmonsUrbanski13}
L.~Fishman, D.~Simmons, and M.~Urbanski, \emph{Diophantine approximation and
  the geometry of limit sets in {G}romov hyperbolic metric spaces},
  arXiv:1301.5630, 2013.

\bibitem[Hat07]{Hattori07}
T.~Hattori, \emph{Some {D}iophantine approximation inequalities and producs of
  hyperbolic spaces}, J. Math. Soc. Japan \textbf{59} (2007), 239--264.

\bibitem[KLW04]{KleinbockLindenstraussWeiss04}
D.~Kleinbock, E.~Lindenstrauss, and B.~Weiss, \emph{On fractal measures and
  {D}iophantine approximation}, Selecta Mathematica \textbf{10} (2004),
  479--523.

\bibitem[KM99]{KleinbockMargulis99}
D.~Kleinbock and G.~Margulis, \emph{Logarithm laws for flows on homogeneous
  spaces}, Invent. Math. \textbf{138} (1999), 451--494.

\bibitem[KT03]{KleinbockTomanov03}
D.~Kleinbock and G.~Tomanov, \emph{Flows on ${S}$-arithmetic homogeneous spaces
  and application to metric {D}iophantine approximation}, Max Planck Institute
  for Mathematics preprints (2003), no.~65, 1--45.

\bibitem[KW15]{KleinbockWeiss15}
D.~Kleinbock and B.~Weiss, \emph{Values of binary quadratic forms at integer
  points and schmidt games}, Recent trends in ergodic theory and dynamical
  systems (Vadodara, 2012), Contemporary Mathematics, vol. 631, Amer. Math.
  Soc., 2015, pp.~77--92 (arXiv:1311.1560).

\bibitem[McM10]{McMullen10}
C.~McMullen, \emph{Winning sets, quasiconformal maps and {D}iophantine
  approximation}, Geom. Func. Anal. \textbf{20} (2010), no.~3, 726--740.

\bibitem[NS14]{NesharimSimmons14}
E.~Nesharim and D.~Simmons, \emph{$\mathbf{Bad}(s, t)$ is hyperplane absolute
  winning}, Acta Arith. \textbf{164} (2014), no.~2, 145--152.

\bibitem[PV05]{PollingtonVelani05}
A.~Pollington and S.~Velani, \emph{Metric {D}iophantine approximation and
  `absolutely friendly' measure}, Selecta Mathematica \textbf{11} (2005),
  297--307.

\bibitem[Sch66]{Schmidt66}
W.~Schmidt, \emph{On badly approximable numbers and certain games}, Trans.
  Amer. Math. Soc. \textbf{123} (1966), 178--199.

\end{thebibliography}

\end{document}